\documentclass[12pt,a4paper]{article}

\usepackage{amsmath,amstext,amssymb,amscd,euscript}
 \usepackage{graphicx}

\usepackage[russian,english]{babel}
\usepackage[cp1251]{inputenc}

\newcommand{\Xcomment}[1]{}
\newtheorem{theorem}{Theorem}[section]
\newtheorem{lemma}[theorem]{Lemma}
\newtheorem{corollary}[theorem]{Corollary}
\newtheorem{prop}[theorem]{Proposition}

\newenvironment{proof}{\noindent{\bf Proof}\/}%
{\hfill$\qed$\medskip}

\def\qed{\Box}

\makeatletter \@addtoreset{equation}{section} \makeatother

\newenvironment{numitem1}{\refstepcounter{equation}\begin{enumerate}%
\item[(\thesection.\arabic{equation})]}{\end{enumerate}}

\newcommand{\refeq}[1]{(\ref{eq:#1})}  

 \makeatletter
\renewcommand{\section}{\@startsection{section}{1}{0pt}%
{-3.5ex plus -1ex minus -.2ex}{2.3ex plus .2ex}%
{\normalfont\Large}}
 \makeatother

 \makeatletter
\renewcommand{\subsection}{\@startsection{subsection}{2}{0pt}%
{-3.0ex plus -1ex minus -.2ex}{1.5ex plus .2ex}%
{\normalfont\normalsize\bf}}
 \makeatother

 \newcommand{\SEC}[1]{\ref{sec:#1}}  
\newcommand{\SSEC}[1]{\ref{ssec:#1}}  

\def\Rset{{\mathbb R}}

\def\Zset{{\mathbb Z}}

\def\Ascr{{\cal A}}
\def\Bscr{{\cal B}}
\def\Cscr{{\cal C}}
\def\Dscr{{\cal D}}
\def\Escr{{\cal E}}

\def\Gscr{{\cal G}}
\def\Hscr{{\cal H}}
\def\Iscr{{\cal I}}

\def\Lscr{{\cal L}}
\def\Mscr{\EuScript{M}}
\def\Pscr{{\cal P}}

\def\Rscr{{\cal R}}
\def\Sscr{{\cal S}}
\def\Tscr{{\cal T}}
\def\Uscr{{\cal U}}

\def\Xscr{{\cal X}}

\def\frakC{\mathfrak{C}}
\def\frakG{\mathfrak{G}}

\def\tilde{\widetilde}
\def\hat{\widehat}
\def\bar{\overline}
\def\eps{\varepsilon}

\def\xmin{x^{\rm min}}
\def\xmax{x^{\rm max}}

\def\bmax{b^{\rm max}}
\def\qmax{q^{\rm max}}

\def\onebf{{\bf 1}}

\def\rest#1{_{\,\vrule height 1.5ex width 0.05em depth 1pt\, #1}}

\begin{document}
\parskip=2pt

\title{A poset representation for stable contracts in a two-sided market generated by integer choice functions }

\author{Alexander V.~Karzanov
\thanks{Central Institute of Economics and Mathematics of
the RAS, 47, Nakhimovskii Prospect, 117418 Moscow, Russia; email:
akarzanov7@gmail.com.}
}
\date{}

 \maketitle
\vspace{-0.7cm}
 \begin{abstract}
Generalizing a variety of earlier problems on stable contracts in two-sided markets, Alkan and Gale introduced in 2003 a general stability model on a bipartite graph $G=(V,E)$ in which the vertices are interpreted as ``agents'', and the edges as possible ``contract'' between pairs of ``agents''. The edges are endowed with nonnegative capacities $b$ giving upper bounds on ``contract intensities'', and the preferencies  of each ``agent'' $v\in V$ depend on a \emph{choice function} (CFs) that acts on the set of ``contracts'' involving $v$, obeying the well motivated axioms  of \emph{substitutability} and \emph{cardinal monotonicity}. In their model, the capacities and choice functions can take reals or discrete values and, extending well-known earlier results on particular cases, they proved that systems of \emph{stable} contracts always exist and, moreover, their set $\Sscr$ constitutes a distributive lattice under a natural comparison relation $\prec$.
 
In this paper, we study Alkan--Gale's model when all capacities and choice functions take integer values. We characterize the set of rotations -- augmenting cycles linking neighboring stable assignments in the lattice $(\Sscr,\prec)$, explain how to construct the rotations efficiently, and devise a weighted poset in which the lattice of closed functions is isomorphic to $(\Sscr,\prec)$, thus obtaining an explicit representation for the latter. We show that in general the size of the poset is at most $\bmax|E|$, where $\bmax$ is the maximal capacity, and the poset can be constructed in pseudo polynomial time. Then we explain that by imposing an additional condition on CFs, the size of the poset becomes polynomial in $|V|$, and the total time reduces to a polynomial in $|V|,\log\bmax$.
 \medskip

 \end{abstract}


\section{Introduction}  \label{sec:intr}

The model due to Alkan and Gale on stable ``schedule matchings'' in two-sided markets was introduced in 2003 in~\cite{AG} as a far generalization of stable marriages studied in the prominent work  by  Gale and Shapley~\cite{GS} and subsequent extensions of ``one-to-many'' and ``many-to-many'' types. In particular, it essentially generalizes the popular stable allocation model by Baiou and Balinski~\cite{BB}. 

Recall that in~\cite{AG} one considers a \emph{bipartite} graph $G=(V,E)$ in which the edges $e\in E$ are endowed with nonnegative real upper bounds, or \emph{capacities} $b(e)\in\Rset_+$. The vertices of $G$ are interpreted as ``agents of the market'', and the edges as possible ``contracts'' between them, of which ``intensities'' are bounded by the capacities. More precisely, in a general setting of~\cite{AG}, for each vertex (``agent'') $v\in V$, a vector of contract values (``intensities'') concerning $v$ can be chosen within a prescribed closed subset $\Bscr_v$ of the box $\{z\in \Rset_+^{E_v}\colon z(e)\le b(e)~ \forall e\in E_v\}$, where $E_v$ is the set of edges of $G$ incident to $v$. The entire box and its integer sublattice are regarded as typical samples of $\Bscr_v$. 

The behavior of ``agent'' $v$ depends on a \emph{choice function} (CF), generalizing a linear order on $E_v$ in simpler models. This is a continuous map $C_v$ of $\Bscr_v$ into itself such that $C_v(z)\le z$ for all $z\in\Bscr_v$. The vectors $z\in\Bscr_v$ satisfying $C_v(z)=z$ are called \emph{acceptable}, and the CF $C_v$ establishes on them the so-called \emph{revealed preference} relation $\prec_v$, defined by $z'\prec_v z\Longleftrightarrow C_v(z\vee z')=z$ (where $z\vee z'$ takes the values $\max\{z(e),z'(e)\}$, $e\in E_v$). For $x\in \Rset_+^E$ and $v\in V$, let $x_v$ denote the restriction of $x$ to $E_v$, and let $\Bscr$ be the set $\{x\in \Rset_+^E \colon \,x_v\in\Bscr_v\; \forall v\in V\}$. A vector (assignment) $x\in \Bscr$ for which all restrictions $x_v$ are acceptable is said to be \emph{stable} if, roughly, there is no edge $e=\{u,v\}$ (``contract between agents $u$ and $v$'') such that one can increase $x$ at $e$ so as to obtain $x'\in\Bscr$ which is better than $x$ for both $u$ and $v$, in the sense that $C_u(x'_u)\succ_u x_u$ and $C_v(x'_v)\succ_v x_v$. 

In the setting of~\cite{AG}, the CFs $C_v$ ($v\in V$) are assumed to obey the axioms of \emph{consistence} and \emph{substitutability} (also called \emph{persistence}), which earlier have been known and well motivated for simpler models, going back to Kelso and Crawford~\cite{KC}, Roth~\cite{roth}, Blair~\cite{blair}, and some others. By adding one more axiom, of \emph{size-monotonicity}, Alkan and Gale succeeded to develop a (mini)theory embracing a variety of stability problems in weighted bipartite graphs. A central result in it, generalizing earlier results for particular cases, is that the set of stable assignments is nonempty and constitutes a \emph{distributive lattice} generated  by the above-mentioned relations $\prec_v$ for vertices $v$ in one side of $G$. (Precise definitions and a brief account on results from~\cite{AG} needed to us will be given in Sect.~\SEC{defin}.) As simple illustrations of their model, Alkan and Gale point out the stable allocation problem and the stable diversification one. We refer to a general case in~\cite{AG} as \emph{AG-model}. 

Following terminology in~\cite{AG}, in a fixed partition of $V$ into two sides (color classes of $G$), one side is denoted by $W$, and the other by $F$, called the sets of \emph{workers} and \emph{firms}, respectively. Let $\Sscr$ denote the set of stable vectors (assignments). For $x,y\in\Sscr$, we write $x\prec_F y$ if $x_f\prec_f y_f$ holds for all $f\in F$; the relation $\prec_W$ is defined symmetrically. The \emph{polarity} established in~\cite{AG} says that $x\prec_F y$ if and only if $y\prec_W x$.

It should be noted that some tempting issues are left beyond discussion in~\cite{AG}, such as: (i) characterize (and/or efficiently construct) augmenting vectors (like rotations or so) determining ``elementary'' transformations in the lattice $(\Sscr,\prec_F)$; or (ii) represent $(\Sscr,\prec_F)$ via the lattice of closed (ideal-wise) functions on an explicitly constructed poset. When $\Sscr$ is finite, such a poset exists by the classical theorem due to Birkhoff~\cite{birk} (though to find it may take  big efforts when the lattice is given implicitly). For a general domain $\Bscr$ and general CFs $C_v$ in the AG-model, both problems (i) and (ii) look highly sophisticated. In some particular cases, e.g. those occurring in \cite{IL,ILG,BB,BAM,DM,FZ}, both problems are solved efficiently, in strongly polynomial time; here the augmenting vectors are related to certain cycles, called \emph{rotations}, which simultaneously serve as elements of the poset in representation of $(\Sscr,\prec_F)$. (Originally the concept of rotation, under the name of ``all-or-nothing cycle'', was introduced by Irving~\cite{irv} in connection with the stable roommates problem.) A different sort of rotations is demonstrated in the work~\cite{karz1} devoted to a sharper version of stable allocation problem where the vertices are endowed with \emph{weak} (rather than strong) linear orders; here both problems are solvable in strongly polynomial time as well, but for integer-valued capacities and quotas, rotations may be formed by rational vectors having exponentially large numerators and denominators in the size of the graph.

In this paper, we consider the integer version of AG-model; namely, we assume that $b\in\Zset_+^E$ and for each $v\in V$, the CF $C_v$ is an arbitrary choice function obeying the above axioms and acting within the integer box $\{z\in\Zset_+^{E_v}\colon z(e)\le b(e)\; \forall e\in E_v\}=:\Bscr_v$. We will refer to the vectors in $\Bscr$ acceptable for all ``agents'' as \emph{generalized matchings}, or \emph{g-matchings} for short, and use the abbreviation \emph{SGMM} for our stable g-matching model. (Note that our study of SGMM was inspired by some ideas from the weakened version~\cite{karz3} where an ``intermediate'' stability model between SGMM and the stable allocation problem is considered.) Regarding computational aspects, we assume that each $C_v$ is given via an \emph{oracle} that, being asked of $z\in\Bscr_v$, outputs the vector $C_v(z)$. 

Our first aim is to characterize the irreducible augmentations in the lattice $(\Sscr,\prec_F)$, by addressing the following problem: given $x\in\Sscr$, find the set $\Sscr_x$ of stable g-matchings $x'$ \emph{immediately succeeding} $x$ in the lattice, i.e. $x\prec_F x'$, and there is no $y\in\Sscr$ between $x$ and $x'$. We solve it by constructing a set $\Rscr(x)$ of \emph{edge-simple} cycles in $G$ (which need not be simple) where each cycle $R$ is associated with a $0,\pm 1$ function $\chi^R$ on $E$ and determines the g-matching $x':=x+\chi^R$ in $\Sscr_x$. This gives a bijection between $\Rscr(x)$ and $\Sscr_x$. The cycles in $\Rscr(x)$ are just what we call the \emph{rotations} applicable to $x$.

For a rotation $R\in\Rscr(x)$, we denote by $\tau_R(x)$ the maximal integer weight such that $x+\lambda\chi^R$ is a stable g-matching for all $\lambda=1,2,\ldots,\tau_R(x)$. Rotations and their maximal weights play an important role in an explicit representation of the lattice $(\Sscr,\prec_F)$. Here we follow an approach originally appeared in Irving and Leather~\cite{IL} (for stable marriages) and subsequently developed in some extensions (e.g. in Bansal et al.~\cite{BAM} for the ``many-to-many'' model, and in Dean and Munshi~\cite{DM} for stable allocations).

More precisely, we construct a sequence $\Tscr=(x^0,x^1,\ldots,x^N)$ of g-matchings along with rotations $R_1,\ldots,R_{N}$ and weights $\tau_1,\ldots,\tau_{N}\in\Zset_{>0}$ such that: $x^1$ and $x^N$ are, respectively, the minimal ($\xmin$) and maximal ($\xmax$) element of $(\Sscr,\prec_F)$, and $x^{i}=x^{i-1}+\tau_i\, \chi^{R_i}$ for $i=1\ldots, N$, where $R_i\in\Rscr(x^{i-1})$ and $\tau_i=\tau_{R_i}(x^{i-1})$ (note that $R_i=R_j$ for $i\ne j$ is possible). We call $\Tscr$ a \emph{full route}. An important fact is that the family of pairs $(R,\tau)$ (with possible repetitions) does not depend on a full route. Moreover, taking the (canonical) family $\Uscr$ of all rotation occurrences in a full route, we arrange a transitive and antisymmetric binary relation $\lessdot$ on it, obtaining a weighted poset $(\Uscr,\tau,\lessdot)$ (where each rotation occurrence $R\in\Uscr$ is taken with the corresponding maximal weight $\tau(R)$). This gives rise to a canonical bijection $\omega$ of $\Sscr$ to the set of closed functions on this poset, just giving the desired representation for $(\Sscr,\prec_F)$. Here a function $\lambda:\Uscr\to\Zset_+$ with $\lambda \le \tau$ is called \emph{closed} if $R\lessdot R'$ and $\lambda(R')>0$ imply $\lambda(R)=\tau(R)$. A closed function $\lambda$ generates the stable g-matching $x=\omega^{-1}(\lambda)$ by setting $x:=\xmin+\sum(\lambda(R)\chi^{R}\colon R\in\Uscr)$. 

Note that there is an essential difference between full routes $x^0,x^1,\ldots,x^N$ occurring in SGMM and in usual stability models, such as the stable allocation one. The behavior for the latter is simpler at two points. First, for each $i$, the weight $\tau_i$ of $R_i$ is computed easily, to be the minimum among the residual capacities $b(e)-x^i(e)$ when $\chi^{R_i}(e)=1$, and the values $x^i(e)$ when $\chi^{R_i}(e)=-1$. Second, all rotations $R_1,\ldots,R_{N}$ are different simple cycles, and $N$ is $O(|E|)$. This implies that the elements of the representing poset can be identified with the rotations and the poset has $O(|E|)$ elements. Moreover, the rotations and their poset can be constructed in strongly polynomial time (e.g. an algorithm in~\cite{DM} has running time $O(|E|^2\log |V|)$).

In case of SGMM, the behavior is more intricate. First, given a rotation $R$ applicable to a stable $x$, to compute the maximal weight $\tau_R(x)$ takes $O(\bmax)$ iterations  in general. Second, in the process of constructing a full route, one and the same rotation can appear several times (even $\Omega(\bmax)$ times). This can be illustrated by an example where the basic graph $G$ has only six vertices, but the size of the poset is of order $\bmax$.
Yet, we can show that in a general case of SGMM, the size is $O(\bmax|E|)$  (in Lemma~\ref{lm:N}). 

As an additional contribution, we explain that ``pseudo polynomial behavior'' can be improved by imposing one more requirement on the choice functions $C_v$, $v\in V$, that we call the \emph{gapless} condition (defined in Sect.~\SEC{special}). It leads to the following property: if $x,x',x''$ are stable g-matchings such that $x\prec_F x'\prec_F x''$ and if a rotation $R$ is applicable to $x$ and to $x''$, then $R$ is applicable to $x'$ as well. (The class of CFs obeying this requirement, in addition to the axioms as above, includes as a special case the CFs generated by linear orders, such as in the stable allocation model. There are more general interesting representatives of this class; however, the limits of this paper do not allow us to discuss this subject here, leaving it to a separate work.)

We show that subject to the gapless condition, all rotations occurring in a full route are different and the number of rotations is estimated as $O(|E|^2|F| |W|)$ (in Theorem~\ref{tm:condC}). This implies that the weighted poset  has a similar size $O(|E|^2|F| |W|)$. We construct this poset in weakly polynomial time (where a factor of $O(\log \bmax)$ arises in the procedure of finding the maximal feasible weight $\tau_R$ of a rotation $R$, based on a ``divide-and-conquer'' technique). (Note that in applications, the availability of such a poset, which has size polynomial in $|V|$, gives rise to solvability, in strongly polynomial time, of the problem of minimizing a linear function over the stable g-matchings, by reducing it to the usual min-cut problem via Picard's method~\cite{pic}, like it was originally proposed for stable marriages of minimum weight in~\cite{ILG}).)

(Concerning computational complexity aspects, we follow standard terminology, such as in~\cite{GJ}, and say that an algorithm that we deal with in our case is \emph{pseudo, weakly, strongly} polynomial if its running \emph{time}, concerning standard operations plus oracle calls, is bounded by a polynomial in $|V|$ and $\bmax$, in $|V|$ and $\log \bmax$, and in $|V|$, respectively. The term \emph{efficient} is applied to weakly and strongly polynomial algorithms. Note that we often do not care of precisely estimating time bounds of procedures that we devise and restrict ourselves merely by establishing their pseudo, weakly, or strongly polynomial-time complexity.
As to the oracles for CFs, we assume their efficiency (as though thinking that one application of $C_v$ takes time polynomial in $|E_v|$). In fact, we will estimate the number of  \emph{oracle calls} rather than the complexity of their implementations, like in a wide scope of problems dealing with oracles.)

In should be noted that in some of our constructions we appeal to ingredients in Birkhoff's (mini)theory in~\cite{birk} on finite distributive lattices and their poset representations. In addition, we explain that if the lattice is given via an oracle that, being asked of an element, outputs its immediate successors, then the poset can be constructed by using $O(|N|^2)$ oracle calls, where $N$ is the length of a maximal chain in the lattice.

This paper is organized as follows.
Section~\SEC{defin} is devoted to basic definitions and backgrounds. In Sect.~\SEC{act-rot} we describe the construction of rotations for our model SGMM and show that they correspond to the immediate succeeding relations in the lattice of stable g-matchings $(\Sscr,\prec_F)$ (Propositions~\ref{pr:xxp} and~\ref{pr:xpy}). Section~\SEC{addit_prop} considers so-called \emph{non-excessive routes}, sequences $(x=x^0,x^1,\ldots,x^N=y)$ of stable g-matchings connecting a fixed pair $x\prec_F y$, where each $x^i$ is obtained from $x^{i-1}$ by applying some rotation $R_i$ taking with maximal possible weight $\lambda_i$. The main fact shown here, in Proposition~\ref{pr:invar_rot} (extending the well-known invariance property of rotations for stable marriages from~\cite{IL})  is that the family $\Pi(x,y)$ of pairs $(R_i,\lambda_i)$ (with possible repetitions) does not depend on the non-excessive route from $x$ to $y$. In Sect.~\SEC{poset_rot}, taking as $x,y$ the minimal ($\xmin$) and maximal ($\xmax$) element in $(\Sscr,\prec_F)$, we turn $\Pi(x,y)$ into the set of elements and their weights in the desired poset $(\Uscr,\tau,\lessdot)$. Here, to construct the partial order $\lessdot$ and prove that the lattice of closed functions in the poset is isomorphic to $(\Sscr,\prec_F)$ (Theorem~\ref{tm:isomorph_complete}), we attract additional arguments (in Sects.~\SSEC{birk}--\SSEC{complete_lat}), in particular, some strengthenings of  constructions and results from Birkhoff~\cite{birk}.

Section~\SEC{comput} estimates the complexity of procedures involved in the process of constructing the poset $(\Uscr,\tau,\lessdot)$. One shows that in a general case of SGMM, the whole algorithm is pseudo polynomial; in particular, the amount of oracle calls for CFs is estimated as $O(|E|^2(\bmax N+N^2))$  (Theorem~\ref{tm:gen_time}), and $N$ is $O(\bmax|E|)$ (Lemma~\ref{lm:N}), where $N$ is the length of a non-excessive (full) route from $\xmin$ to $\xmax$. In Sect.~\SEC{special} we introduce the above-mentioned \emph{gapless condition} on the choice functions and show that our methods become efficient in this case. More precisely, one shows (in Theorem~\ref{tm:condC} and assertion~\refeq{complex_C}) that: any rotation occurs in a non-excessive route at most once; the number of rotations and the length $N$ of a full route is $O(|E|^2 |F| |W|)$; and the poset $(\Uscr,\tau,\lessdot)$ can be constructed in weakly polynomial time, namely, in $\log \, \bmax$ times a polynomial in $|V|$. A nontrivial task on this way is to find the initial element $\xmin$; we elaborate a weakly polynomial algorithm to solve this task. In concluding Sect.~\SEC{concl} we briefly outline two issues: finding a stable g-matching of minimum total cost when the poset $(\Uscr,\tau,\lessdot)$ is available, and an example of a ``small'' graph $G$ having a ``big'' poset representation (following~\cite{karz3}).


\section{Definitions and settings} \label{sec:defin}

In what follows we often use terminology and definitions from~\cite{AG}. 

We consider a bipartite graph $G=(V,E)$ in which the vertex set $V$ is
partitioned into two parts (independent sets, color classes) $W$ and $F$, conditionally called the sets of \emph{workers} and \emph{firms}, respectively. The edges $e\in E$ are endowed with nonnegative integer \emph{capacities}
$b(e)\in\Zset_{+}$.
One may assume, without loss of generality, that the graph $G$ is connected and
has no multiple edges. Then $|V|-1\le |E|\le \binom{|V|}{2}$. The
edge connecting vertices $w\in W$ and $f\in F$ may be denoted as $wf$.

For a vertex $v\in V$, the set of its incident edges is denoted by $E_v$. We write $\Bscr=\Bscr^{(b)}$ for the nonnegative \emph{integer} box $\{x\in \Zset_+^E\colon x\le b\}$. For $v\in V$, the restriction of $\Bscr$ to the set $E_v$ is denoted by $\Bscr_v$.
\medskip

$\bullet$ (\textbf{choice functions}) Each vertex (``agent'') $v\in V$ can prefer one vector in $\Bscr_v$ to another. The preferences depend on a \emph{choice function} (CF) related to $v$. This is a map $C=C_v$ of $\Bscr_v$ into itself satisfying $C(z)\le z$ for all $z\in \Bscr_v$. Also $C$ obeys additional axioms. Two of them consider $z,z'\in\Bscr_v$ and require that:
  \begin{itemize}
\item[(A1)] if $z\ge z'\ge C(z)$, then $C(z')=C(z)$ (\emph{consistence});
\item[(A2)] if $z\ge z'$, then $C(z)\wedge z'\le C(z')$ (\emph{substitutability}, or
\emph{persistence}).
  \end{itemize}

\noindent(Hereinafter, for $a,b\in\Rset^S$, the functions $a\wedge b$ and $a \vee b$  take the values $\min\{a(e),b(e)\}$ and $\max\{a(e),b(e)\}$, $e\in S$, respectively.)

In particular, axiom (A1) implies that any $z\in\Bscr_v$ satisfies $C(C(z))=C(z)$. Also (A1) and (A2) imply that $C$ is \emph{stationary} (by terminology of~\cite{AG}), which means that
 \begin{numitem1} \label{eq:plott}
any $z,z'\in\Bscr_v$ satisfy $C(z\vee z')=C(C(z)\vee z')$.
  \end{numitem1}

\noindent(In case of boolean assignments (taking values 0 and 1), it is known as the  property of \emph{path independence} introduced by Plott~\cite{plott}.)
One more well-known axiom imposed on $C=C_v$ for each $v\in V$ is called the \emph{size monotonicity} condition; it requires that
 \begin{itemize}
 \item[(A3)] if $z\ge z'$, then $|C(z)|\ge |C(z')|$.
 \end{itemize}
 
\noindent(Hereinafter, for a numerical function $a$ on a finite set $S$, we write $a(S)$ for $\sum(a(e) : e\in S)$, and $|a|$ for $\sum(|a(e)|\colon e\in S)$. In particular,
$a(S)=|a|$ if $a$ is nonnegative.)

Note that one easily shows that (A2) and~(A3) imply (A1); yet, it is convenient to us to keep~(A1) as well. A special case of~(A3) is the condition of \emph{quota filling}; it is applied when a vertex $v\in V$ is endowed with a \emph{quota} $q(v)\in \Zset_+$  and is viewed as:
 \begin{itemize}
 \item[(A4)]  any $z\in\Bscr_v$ satisfies $|C(z)|=\min\{|z|,q(v)\}$.
 \end{itemize}

\noindent\textbf{Example 1.} 
A popular instance of choice functions $C_v$ obeying axioms (A1),(A2) and (A4) is generated by a \emph{linear order} $>_v$ on the set $E_v$. Here
for $e,e'\in E_v$ with $e>_w e'$, the ``agent'' $v$ is said to prefer the edge (``contract'') $e$ to $e'$. Then $C_v$ is defined by the following rule related to  the order $>_v$ and a quota $q(v)\in\Zset_+$: for $z\in\Bscr_v$, if $|z|\le q(v)$, then $C_v(z):=z$; and if $|z|> q(v)$, then, renumbering the edges in $E_v$ as $e_1,\ldots,e_{|E_v|}$ so that $e_i>_v e_{i+1}$ for each $i$, take the maximal $j$ satisfying $r:=\sum(z(e_i)\colon i\le j)\le q(v)$ and define $C_v(z)$ to be $(z(e_1),\ldots,z(e_j),q(v)-r,0,\ldots,0)$.
 \medskip

$\bullet$ (\textbf{stability}) For a vertex $v\in V$ and a function $x$ on $E$, let $x_v$ denote the restriction of $x$ to  $E_v$. A vector (function, assignment) $z\in \Bscr_v$ is called  \emph{acceptable} if $C_v(z)=z$; the collection of such vectors is
denoted by $\Ascr_v$. This notion is extended to $\Bscr$; namely, we
say that $x\in\Bscr$ is (globally) acceptable if $x_v\in\Ascr_v$ for all $v\in V$. The collection of acceptable vectors in $\Bscr$ is denoted by $\Ascr$.

For any $v\in V$, the CF $C_v$ establishes preference relations on acceptable functions on $E_v$; namely (cf.~\cite{AG}), $z\in\Ascr_v$ is said to be \emph{revealed preferred} to $z'\in\Ascr_v-\{z\}$ if
   \begin{equation} \label{eq:zzp}
   C_v(z\vee z')=z,
   \end{equation}
denoted as $z'\prec_v z$. The relation $\prec_v$ is transitive (which is shown by use of~\refeq{plott}).

The preferences between acceptable functions on the sets $E_v$, $v\in
V$, are extended, in a natural way, to the whole $E$. Namely, for the side $F$ of $G$ and distinct $x,y\in \Ascr$, we write $x\prec_F y$ if $x_f\preceq_f y_f$ holds for all $f\in F$. The preferences in $\Ascr$ relative to the
``workers'' are defined in a similar way and denoted as $\prec_{\,W}$.

Given $G$, $b$  and $C_v$ ($v\in V$) as above, we will refer to functions $x\in\Ascr$ on $E$ as (acceptable) \emph{generalized matchings}, or \emph{g-matchings} for short. 
 \medskip

\textbf{Definition 1.} Let $v\in V$ and $z\in\Ascr_v$. We say that an edge $e\in E_v$ is \emph{interesting} for $v$ under $z$ if there exists $z'\in\Bscr_v$ such that
  \begin{equation} \label{eq:inter_e}
  z'(e)>z(e),\quad \mbox{$z'(e')=z(e')$ for all $e'\ne e$,}\quad \mbox{and $C_v(z')(e)>z(e)$}
  \end{equation}
(the term ``interesting'' is justified by an expectation that an increase at $e$ gives rise to a better assignment for $v$). Extending this to g-matchings on $E$, we say that an edge $e=wf\in E$ is \emph{interesting} for a vertex (``agent'') $v\in\{w,f\}$ under a g-matching $x\in\Ascr$ if so is for $v$ and $x_v$. If  $e=wf\in E$
is interesting under $x$ for both vertices $w$ and $f$, then the edge $e$ is
called \emph{blocking} $x$. A g-matching $x\in \Ascr$ is called
\emph{stable} if no edge in $E$ blocks $x$. The set of stable g-matchings 
is denoted by $\Sscr=\Sscr_{G,b,C}$.
\medskip

$\bullet$ ~For $v\in V$, the set $\Ascr_v$ endowed with the preference relation
$\succ_v$ turns into a \emph{lattice}. In this lattice, for $z,z'\in\Ascr_v$, the
least upper bound (join) $z\curlyvee z'$ is expressed as $C_v(z\vee z')$, and
the greatest lower bound (meet) $z\curlywedge z'$ is expressed as $C_v(\bar z\wedge \bar z')$, where $\bar y$ is the \emph{closure} of $y\in\Ascr_v$, defined as $\sup(y'\in\Bscr_v\colon C_v(y')=y)$ (the supremum is well defined since for $y',y''\in \Bscr_v$,  if $C_v(y')=C_v(y'')=y$, then $C_v(y'\vee y'')=y$, by applying~\refeq{plott}). For details, see~\cite{AG}.
 \medskip

$\bullet$ ~In a general setting, the AG-model of~\cite{AG} deals with a bipartite graph $G=(V,E)$ and domains $\Bscr_v$ ($v\in V$) that are closed subsets of the real boxes $\{z\in \Rset_+^{E_v}\colon z\le b\rest{E_v}\}$ for some $b\in \Rset_+^E$. Also one assumes that each choice function $C_v$ ($v\in V$) acting on $\Bscr_v$ is continuous and obeys axioms as above. So our model of stable g-matchings is an ``integer version'' of that; we denote it as \emph{SGMM}. General results in~\cite{AG} imply that the set $\Sscr$ of stable g-matchings in SGMM possesses the following nice properties that will be important to us in what follows:
 \begin{numitem1} \label{eq:AG}
 \begin{itemize}
\item[(a)] $\Sscr$ is nonempty, and  $(\Sscr,\prec_F)$ is a
\emph{distributive} lattice (cf.~\cite[Ths.~1,8]{AG});
\item[(b)]
(\emph{polarity}): $\prec_F$ is opposite to 
$\prec_W$ on $\Sscr$; namely: for $x,y\in\Sscr$, if  $x_f\prec_f y_f$ for all $f\in F$, then $y_w\prec_w x_w$ for all $w\in W$, and vice versa (cf.~\cite[Cor.~2]{AG});
\item[(c)]
(\emph{unisizeness}): for each vertex $v\in V$, the size $|x_v|$ is
the same for all stable g-matchings $x\in \Sscr$ (cf.~\cite[Th.~6]{AG}).
 \end{itemize}
   \end{numitem1}

We denote the minimal and maximal elements in the lattice $(\Sscr,\prec_F)$ by
$\xmin$ and $\xmax$, respectively; then the former is the best and the latter
is the worst for the part $W$, in view of the polarity~(2.3)(b).

In what follows, when no confuse can arise, we may write $\prec$ for
$\prec_F$.   Also for $v\in V$ and $e\in E_v$, we write $\onebf^e_v$ for the
unit base vector of $e$ in $\Rset^{E_v}$ (taking value 1 on $e$, and 0
otherwise). We often will omit the term $v$ if it is clear from the context, writing $\onebf^e$ for $\onebf^e_v$.


 \section{Active graph and rotations} \label{sec:act-rot}

Throughout this section, we fix a stable g-matching $x\in\Sscr$ different
from $\xmax$. We are interested in characterizing the set $\Sscr_x$ of stable g-matchings $x'$ that are \emph{close} to $x$ and satisfy $x'\succ_F x$. This means that $x$ \emph{immediately precedes} $x'$ in the lattice $(\Sscr,\prec_F)$. We shall show that the set $\Sscr_x$ has cardinality $O(|E|)$ and can be found efficiently. To this aim, we construct the so-called \emph{active graph} and then extract in it special cycles called \emph{rotations}. 

We call an edge $e\in E$ \emph{saturated} (by $x$) if $x(e)=b(e)$, and unsaturated otherwise.

For our purposes some equivalent definitions of ``interesting'' edges will be of use.
 \begin{lemma} \label{lm:e_interest}
 
Let $v\in V$,  $e\in E_v$, $z\in\Ascr_v$, and $z(e)<b(e)$. The following are equivalent:

{\rm(i)} $e$ is interesting for $v$ under $z$;

{\rm(ii)} there is $\tilde z\in\Ascr_v$ such that $\tilde z\succ_v z$ and $\tilde z(e)>z(e)$;

{\rm(iii)} $C_v(z+\onebf^e)\ne z$ (where $\onebf^\bullet$ stands for $\onebf^\bullet_{v}$);

{\rm(iv)} $C_v(z+\onebf^e)$ is either {\rm(a)} $z+\onebf^e$, or {\rm(b)} $z+\onebf^e-\onebf^{e'}$ for some $e'\in E_v-\{e\}$.
  \end{lemma}

  \begin{proof}
(i)$\to$(ii). Let $e$ be interesting for $v$ under $z$ and let $z'$ be as in~\refeq{inter_e}. Define $\tilde z:=C(z')$, where $C:=C_v$. Then $\tilde z(e)>z(e)$. Since $z'>z$ implies $z'\vee z=z'$, we have: 
   $$
   \tilde z=C(z')=C(z'\vee z)=C(C(z')\vee z)=C(\tilde z\vee z),
   $$
(using~\refeq{plott}), whence $\tilde z\succ_v z$. 

(ii)$\to$(iii). Let $z':=z+\onebf^e$ and let $\tilde z$ be as in~(ii). Then $\tilde z\vee z\ge z'$. Applying~(A2), we have $C(\tilde z\vee z)(e)\wedge z'(e)\le C(z')(e)$. This implies $C(z')(e)>z(e)$ (in view of $C(\tilde z\vee z)=\tilde z$). So $C(z')\ne z$.

(iii)$\to$(iv).  Let $z':=z+\onebf^e$. Then the relations $C(z')\le z'$, $C(z')\ne z$, $|C(z')|\ge |z|$ (by~(A3) for $z'>z$), and the integrality of $C(z')$ imply that only two cases for $C(z')$ are possible, namely, those pointed out in~(iv).

(iv)$\to$(i). Taking $z':=z+\onebf^e$ and using~(iv), we obtain~\refeq{inter_e}.
  \end{proof}

As to ``non-interesting'' edges, we will use the following:
  \begin{numitem1} \label{eq:e_noninterest}
for $v\in V$ and $z\in\Ascr_v$, if an edge $e\in E_v$ is not interesting for $v$ under $z$ and if $\eps\in\Zset_+$ is such that $z(e)+\eps\le b(e)$, then $C_v(z+\eps\onebf^e)=z$. 
  \end{numitem1}
This is immediate for $\eps=0$. For $\eps=1$, we appeal to the equivalence of~(i) and~(iii) in Lemma~\ref{lm:e_interest}.
And for $\eps\ge 1$, taking  $z':=z+\onebf^e$ and $z'':=z+\eps\onebf^e$ and applying~(A2) to $z''\ge z'$, we have $C_v(z'')\wedge z'\le C_v(z')=z$. Then $C_v(z'')(e)\le z(e)$, whence $C_v(z'')=z$.


\subsection{Active graph} \label{ssec:act_graph}

The rotations for $x$ are extracted from a certain subgraph of $G$, called the \emph{active graph} related to $x$. To construct it, we need some definitions and explanations.
 \medskip
 
\noindent\textbf{Definition 2.}
For $f\in F$, define $U_f^+=U_f^+(x)$ to be the set of edges in $E_v$ interesting for $f$ under $x$; in particular, these edges are unsaturated. Also define $U_f^-=U_f^-(x)$ to be the set of edges $e\in E_f$ with $x(e)>0$ that are not interesting for $f$ under $x$. The union of sets $U_f^+$ ($U_f^-$) over all ``firms'' $f\in F$ is denoted by $U^+_F=U^+_F(x)$ (resp. $U^-_F=U^-_F(x)$). By Lemma~\ref{lm:e_interest}(iv), for an edge $a=wf\in U_f^+$, the vector $C_f (x_f+\onebf^a)$ is either $x_f+\onebf^a$ or $x_f+\onebf^a-\onebf^c$ for some $c\in E_f-\{a\}$ (where $\onebf^\bullet$ stands for $\onebf^\bullet_{f}$). In the latter case, we call $(a,c)$ a \emph{legal $f$-pair} (under $x$), or a \emph{legal $F$-pair passing $f$}.
 \medskip

Note that there may be several $f$-legal pairs with the same ``negative'' element, i.e. legal $f$-pairs $(a,c)$ and $(a',c')$ with $c=c'$ and $a\ne a'$ are possible.
  \begin{lemma} \label{lm:acd}
For $f\in F$, let $(a,c)$ be a legal $f$-pairs. Then {\rm(i)}~$c\in U_f^-(x)$ and {\rm(ii)}~$x_f\prec_f x_f+\onebf^a-\onebf^c$. Also {\rm(iii)}~$c$ is not interesting for $f$ under $x_f+\onebf^a-\onebf^c$.
  \end{lemma}
\begin{proof}
~Clearly $x(c)>0$. Suppose, for a contradiction, that $c$ is interesting for $f$ under $x$. Then, by Lemma~\ref{lm:e_interest}(iv), $C(x_f+\onebf^c)$ is either (a) $x_f+\onebf^{c}$, or (b) $x_f+\onebf^{c}-\onebf^{d}$ for some $d\in E_f-\{c\}$, where $C:=C_f$ and $\onebf^\bullet$ stands for $\onebf^\bullet_{f}$. 

In case~(b), take the vectors $z:=x_f+\onebf^{a}$, $z':=x_f+\onebf^{c}$, and $y:=x_f+\onebf^{a}+\onebf^{c}$. Then $C(z)=x_f+\onebf^{a}-\onebf^{c}$ and $C(z')=x_f+\onebf^{c}-\onebf^{d}$. By axiom~(A2) applied to $y>z$ and $y>z'$, we have $C(y)\wedge z\le C(z)$ and $C(y)\wedge z'\le C(z')$. 

It follows that $C(y)(c)\le x(c)-1$ and $C(y)(d)\le x(d)-1$. Also $C(y)(a)\le x(a)+1$ and $C(y)(e)\le x(e)$ for all $e\ne a,c,d$. Therefore, $|C(y)|\le |x_f|-1$. But then $|C(y)|<|x_f|=|C(z)|$, contradicting axiom~(A3) applied to $y>z$.

And in case~(a), for $z,y$ as above, we obtain from $C(y)\wedge z\le C(z)=x_f+\onebf^a-\onebf^c$ that $|C(y)|\le |x_f|$. But $|C(x_f+\onebf^c)|=|x_f|+1$, contradicting~(A3) for $y>x_f+\onebf^c$. 

Next, (ii) in the lemma follows from $C(x_f\vee (x_f+\onebf^a-\onebf^c))=C(x_f+\onebf^a)=x_f+\onebf^a-\onebf^c$. And~(iii) follows from $C(x_f+\onebf^a-\onebf^c+\onebf^c)=C(x_f+\onebf^a)=x_f+\onebf^a-\onebf^c$.
 \end{proof}

Now consider vertices in $W$. We involve certain edges in legal pairs for ``workers'';  their construction is somewhat different from that of legal pairs for ``firms''. 
 \medskip
 
\noindent\textbf{Definition 3.}
Let $c=wf\in U^-_F$ (where $w\in W$). Consider an edge in $a\in E_w\cap U^+_F$; then $x(a)<b(a)$. If the vector $z:=x_w+\onebf^a_w-\onebf^c_w$ is acceptable, i.e. $C_w(z)=z$, then we say that $(c,a)$ forms a \emph{legal $w$-pair}, or a \emph{legal $W$-pair passing $w$}. 
\medskip

Such a pair $(c,a)$ possesses a number of useful properties. First of all, since $x$ is stable and the edge $a=wf$ is interesting for $f$ under $x$, this edge is not interesting for $w$, i.e. $C_w(x_w+\onebf^a)=x_w$, letting $\onebf^\bullet=\onebf_w^\bullet$. Note also that
\begin{numitem1} \label{eq:leg_w}
for a legal $w$-pair $(c,a)$, the vector $z:=x_w+\onebf^a-\onebf^c$ satisfies: (a) $c$ is interesting for $w$ under $z$, and (b) $z\prec_w x_w$.
  \end{numitem1}
Indeed, we have $C_w(z+\onebf^c)=C_w(x_w+\onebf^a)=x_w=z+\onebf^c-\onebf^a$, yielding~(a). And $C_w(z\vee x_w)=C_w(x_w+\onebf^a)=x_w$ implies~(b).
\smallskip

Now fix a legal $w$-pair $(c,a)$ and consider an edge $d\in U^+_F\cap E_w$ different from $a$. It is not interesting for $w$ under $x$ (similar to $a$). The next fact is of importance.
 \begin{lemma} \label{lm:essent_pair}
Let $d\in U^+_F\cap E_w-\{a\}$ be interesting for $w$ under $z:=x_w+\onebf^a-\onebf^c$. Then $(c,d)$ is a legal $w$-pair and $(x_w+\onebf^d-\onebf^c)\succ_w z$.
  \end{lemma}
  \begin{proof}
~Let $C:=C_w$. Since $d$ is interesting under $z$, the vector $C(z+\onebf^d)$ is equal to one of the following: (a) $z+\onebf^d$; (b) $z+\onebf^d-\onebf^{d'}$ for some $d'\ne a$; or (c) $z+\onebf^d-\onebf^a$.

Note that $C(x_w+\onebf^a)=x_w$ and $C(x_w+\onebf^d)=x_w$ imply that the vector $\hat z:=x_w+\onebf^a+\onebf^d$ satisfies $C(\hat z)=x_w$. (For $\hat z>x_w+\onebf^a$ implies (by~(A2)) that $C(\hat z)\wedge(x_w+\onebf^a)\le C(x_w+\onebf^a)=x_w$, whence $C(\hat z)(a)\le x_w(a)$. Similarly, $C(\hat z)(d)\le x_w(d)$, and now $C(\hat z)=x_w$ follows from~(A3) for $\hat z>x_w$.)

In case~(a), $\hat z>z+\onebf^d$ implies $|C(\hat z)|\ge|C(z+\onebf^d)|$, by~(A3). But $|C(\hat z)|=|x_w|$ and $|C(z+\onebf^d)|= |z+\onebf^d|=|x_w|+1$; a contradiction.

In case~(b), $\hat z>z+\onebf^d$ and $C(z+\onebf^d)=z+\onebf^d-\onebf^{d'}$ imply $C(\hat z)\wedge(z+\onebf^d)\le z+\onebf^d-\onebf^{d'}$, by~(A2). But $C(\hat z)(d')=x_w(d')\ge z(d')$ and $(z+\onebf^d)(d')=z(d')$, whereas $(z+\onebf^d-\onebf^{d'})(d')<z(d')$ (since $d'\ne a,d$); a contradiction.

Finally, in case (c), we have 
$$
C(z+\onebf^d)=z+\onebf^d-\onebf^a=(x_w+\onebf^a-\onebf^c)+\onebf^d-\onebf^a
=x_w+\onebf^d-\onebf^c,
$$
implying that $(c,d)$ is a legal $w$-pair. Also $C(z\vee(x_w+\onebf^d-\onebf^c))=C(z+\onebf^d)=x_w+\onebf^d-\onebf^c$, whence $(x_w+\onebf^d-\onebf^c)\succ_w z$, as required.  
  \end{proof}

For $w\in W$ and $c\in E_w\cap U^-_F$, let $U^+[c]$ denote the set of edges $a\in U^+_F\cup E_w$ such that $(c,a)$ is a legal $w$-pair. Preference relations on such pairs established in Lemma~\ref{lm:essent_pair} provides us with the following important property:
  \begin{numitem1}\label{eq:U+a}
for $c$ as above, if $U^+[c]\ne\emptyset$, then there exists a legal $w$-pair $(c,a)$ such that no $d\in U^+_F\cap E_w-\{a\}$ is interesting for $w$ under $x_w+\onebf^a-\onebf^c$; moreover, such a pair $(c,a)$ is \emph{unique}.
  \end{numitem1}
Here the uniqueness follows from the equality $C_w(x_w+\onebf^a-\onebf^c+\onebf^d)=x_w+\onebf^a-\onebf^c$ for each $d\in U^+[c]-\{a\}$, which implies that $a$ is interesting for $w$ under $x_w+\onebf^d-\onebf^c$. 
\medskip

\noindent\textbf{Definition 4.} The legal $w$-pair $(c,a)$ as in~\refeq{U+a} is called \emph{essential} for $c$ under $x$. 
 \medskip

Note that $a\in U^+_F\cap E_w$ may belong to several essential pairs, i.e. essential (legal) $w$-pairs $(c,a)$ and $(c',a')$ with $c\ne c'$ and $a=a'$ are possible.
Using the above observations, we now form an auxiliary directed graph $\Dscr=\Dscr(x)$ as follows:
  \begin{numitem1} \label{eq:Dscr}
  \begin{itemize}
\item[(a)] 
each edge $a=wf\in U^+_F$ generates two vertices $w^a$ and $f^a$ and one directed edge $e^a=(w^a,f^a)$ (from $w^a$ to $f^a$) in $\Dscr$;
 \item[(b)] each edge $c=wf\in U^-_F$ generates vertices $w^c$ and $f^c$ and edge $e^c=(f^c,w^c)$;
 \item[(c)] for $f\in F$, each legal $f$-pair $(a,c)$ generates edge $(f^a,f^c)$;
 \item[(d)] for $w\in W$, each essential $w$-pair $(c,a)$ generates edge $(w^c,w^a)$.
  \end{itemize}
  \end{numitem1}
  
An important property of $\Dscr$ is that 
each vertex of $\Dscr$ has at most one leaving edge.

Indeed: (i) for $f\in F$ and $c\in E_f\cap U^-_F$, the only leaving edge for $f^c$ is $e^c$; (ii) for $w\in W$ and $c\in E_w\cap U^-_F$, the vertex $w^c$ has no leaving edge if $U^+[c]=\emptyset$, and has the only leaving edge $(w^c,w^a)$ otherwise, where $(c,a)$ is the essential $w$-pair for $c$; (iii) for $w\in W$ and $a\in E_w\cap U^+_F$, the only leaving edge for $w^a$ is $e^a$; and (iv) for $f\in F$ and $a\in E_f\cap U^+_F$, the vertex $f^a$ has the only leaving edge $(f^a,f^c)$ if $a$ occurs in the legal $f$-pair $(a,c)$, and has no leaving edge otherwise (when $C_f(x_f+\onebf^a)=x_f+\onebf^a$).
 \smallskip
 
We apply to $\Dscr$ the following
  \begin{description}
\item[\emph{Cleaning procedure}:]
It scans the vertices of $\Dscr$ and updates $\Dscr$ step by step. When scanning a vertex $v$ of $\Dscr$, if we observe that $v$ has no entering edge, then $v$ is deleted from $\Dscr$ (together with the edge leaving $v$ if exists). Along the way, we remove from $\Dscr$ each isolated (zero degree) vertex whenever it appears. Repeat the procedure with the updated $\Dscr$, and so on, until  $\Dscr$ stabilizes.
  \end{description}
  
Upon termination, if the final $\Dscr$ is nonempty, each vertex in it has an entering edge. This and the fact that each vertex of $\Dscr$ has at most one leaving edge imply that each component of $\Dscr$ is a simple directed cycle. 
 \medskip
 
\noindent\textbf{Definition 5.}
We associate to the final $\Dscr$ its natural image $\Gamma=\Gamma(x)$ in $G$, which is a subgraph of $G$ with a $\pm 1$ labeling on the edges, defined as follows: an edge $(w^c,w^a)$  (resp. $(f^a,f^c)$) of $\Dscr$ corresponds to the vertex $w$ (resp. $f$) of $\Gamma$; an edge $e^a$ of $\Dscr$, where $a=wf\in U^+_F$, corresponds to the edge $a$ labeled $+1$ and called \emph{positive}; and an edge $e^c$, where $c=wf\in U^-_F$, corresponds to the edge $c$ labeled $-1$ and called \emph{negative}. We call $\Gamma$ the \emph{active graph} for $x$. The edges of $\Gamma$ are endowed with the corresponding pairing formed by legal $F$-pairs $(a,c)$ and essential $W$-pairs $(c,a)$ used in $D$, that we call \emph{tandems} in $G$. The sets of vertices and edges of $\Gamma$ are denoted by $V_\Gamma$ and $E_\Gamma$, respectively. We set $W_\Gamma:=V_\Gamma\cap W$ and $F_\Gamma:=V_\Gamma\cap F$.
\medskip

The active graph $\Gamma=\Gamma(x)$ is decomposed into pairwise edge disjoint cycles that are images of the components (simple cycles) of $\Dscr$, where the positive and negative edges alternate. Equivalently, each cycle $R=(v_0,e_1,v_1,\ldots,e_k,v_k=v_0)$ of this sort in $\Gamma$ is determined by the tandem relations, namely: for each $i=1,\ldots,k$, the pair $(e_i,e_{i+1})$ forms a tandem, letting $e_{k+1}:=e_1$. Although all edges in $R$ are different, $R$ may be self-intersecting in vertices, i.e. $R$ is edge-simple but not necessarily simple. Also a vertex of $\Gamma$ may be shared by several such cycles.  Depending on the context,  the cycle $R$ may also be regarded as a graph and denoted as $R=(V_R,E_R)$.

Let $\Rscr=\Rscr(x)$ be the set of above-mentioned cycles in $\Gamma$; they are just what we call the \emph{rotations} applicable to $x$. Since the rotations are pairwise edge disjoint, $|\Rscr|\le |E|/4$. For each $R\in\Rscr$, define $R^+$ and $R^-$ to be the sets of positive and negative edges in $R$, respectively. We denote by $\chi^R$ the corresponding incidence $0,\pm 1$ vector in $\Rset^E_+$, taking value 1 for each $e\in R^+$, $-1$ for each $e\in R^-$, and 0 otherwise. We also may write $\chi^{R^+}$ and $\chi^{R^-}$ for the corresponding $0,1$ vectors (so that $\chi^R=\chi^{R^+}-\chi^{R^-}$).


\subsection{Key properties of rotations} \label{ssec:key_prop}

They are exhibited in the following two propositions.
 \begin{prop} \label{pr:xxp}
~For each rotation $R\in\Rscr(x)$, the g-matching $x':=x+\chi^R$ is
stable and satisfies $x\prec_F x'$.
 \end{prop}

We say that $x'$ is obtained from $x$ \emph{by applying the rotation
$R$ with weight}~1, or \emph{by shifting along} $R$,  and denote the set of these g-matchings $x'$ over all $R\in\Rscr(x)$ by $\Sscr_x$. 
 \begin{prop} \label{pr:xpy}
Let $y\in \Sscr$ and $x\prec_F y$. Then there exists $x'\in\Sscr_x$ such that
$x'\preceq_F y$.
  \end{prop}

These two propositions imply that $\Sscr_x$ is just the set of immediately successors of $x$ in $(\Sscr,\prec_F)$, as mentioned in the beginning of this section. The proofs of these propositions will rely on the next two lemmas.

\begin{lemma} \label{lm:acac}
Let $f\in F$. Let $(a(1),c(1)), \ldots,(a(k),c(k))$ be pairwise disjoint legal $f$-pairs in $G$. Then for any $I\subseteq \{1,\ldots,k\}=:[k]$,
  \begin{multline*}
  C_f(x_f+\onebf^{a(1)}+\cdots+\onebf^{a(k)}-\sum(\onebf^{c(i)}\colon i\in I\}) \\
   =x_f+\onebf^{a(1)}+\cdots+\onebf^{a(k)}-\onebf^{c(1)}-\cdots-\onebf^{c(k)}.
  \end{multline*}
  \end{lemma}

 \begin{proof}
~Denote $x_f+\onebf^{a(1)}+\cdots+\onebf^{a(k)}-\sum(\onebf^{c(i)}\colon i\in
I\})$ by $z^I$. We have to show that $C(z^I)=z^{[k]}$ for any
$I\subseteq[k]$, where $C:=C_f$.

First we show this for $I=\emptyset$. To this aim, we compare the actions of
$C$ on $z^\emptyset=x_f+\onebf^{a(1)}+\cdots+\onebf^{a(k)}$ and on
$y^i:=x_f+\onebf^{a(i)}$ for an arbitrary $i\in[k]$. The definition of
$(a(i),c(i))$ implies $C(y^i)=x_f+\onebf^{a(i)}-\onebf^{c(i)}$. Applying~(A2) to $z^\emptyset\ge y^i$, we have
  $$
  C(z^\emptyset)\wedge y^i\le C(y^i)=x_f+\onebf^{a(i)}-\onebf^{c(i)}.
  $$
This and $y^i(c(i))=x(c(i))$ imply $C_f(z^\emptyset)(c(i))<x(c(i))$.

Thus, $C(z^\emptyset)\le z^\emptyset-\onebf^{c(1)}-\cdots-\onebf^{c(k)}=z^{[k]}$. Here the inequality must turn into equality, since axiom~(A3) applied to $z^\emptyset\ge x_f$ gives $|C(z^\emptyset)|\ge|C(x_f)|=|x_f|=|z^{[k]}|$.

Now consider an arbitrary $I\subseteq[k]$. Then $z^\emptyset\ge z^I\ge
C(z^\emptyset)=z^{[k]}$. Applying axiom~(A1), we obtain $C(z^I)=z^{[k]}$,
as required.
 \end{proof}

This lemma implies that
 \begin{numitem1} \label{eq:zxab}
the vector $z:=x_f+\onebf^{a(1)}+\cdots+\onebf^{a(k)}-\onebf^{c(1)}-\cdots-\onebf^{c(k)}$ ($=z^{[k]}$) is acceptable for $f$ and satisfies $z\succ_f x_f$.
  \end{numitem1}
(Here $z\succ_f x_f$ follows from $z=C_f(z^\emptyset\vee x_f)=C_f(C_f(z^\emptyset)\vee x_f)=C_f(z\vee x_f)$.)

The second lemma concerns $W$-pairs.
  \begin{lemma} \label{lm:partW}
Let $w\in W$. Let $(c,a)$ and $(c',a')$ be disjoint essential w-pairs under $x$. Let $z:=x_w+\onebf^a-\onebf^c$ and $z':=x_w+\onebf^{a'}-\onebf^{c'}$. Then: {\rm(i)} the vectors $z+\onebf^{a'}-\onebf^{c'}$ and $z'+\onebf^a-\onebf^c$ are acceptable; {\rm(ii)} no edge $d\in U^+_F(x)$ is interesting for $w$ under $y:=x_w+\onebf^a-\onebf^c+\onebf^{a'}-\onebf^{c'}$; and {\rm(iii)} the pair $(c',a')$ is essential under $z$, and $(c,a)$ is essential under $z'$.
  \end{lemma}
  \begin{proof}
~Let $C:=C_w$. Since $(c,a)$ is essential under $x$, the edge $a'$ is not interesting under $z$, i.e. $C(z+\onebf^{a'})=z$ (cf.~\refeq{U+a}). Similarly, $a$ is not interesting under $z'$, i.e. $C(z'+\onebf^{a})=z'$. Assertion~(i) means that 
 \begin{equation} \label{eq:zzp_acc}
  C(z+\onebf^{a'}-\onebf^{c'})=z+\onebf^{a'}-\onebf^{c'}\quad\mbox{and} \quad
      C(z'+\onebf^a-\onebf^c)=z'+\onebf^a-\onebf^c.
      \end{equation}
      
To show this, apply~(A2) to the inequality $z+\onebf^{a'}>y$; then $C(z+\onebf^{a'})\wedge y\le C(y)$. This implies $C(y)(a)\ge x(a)+1$ and $C(y)(c)\ge x(c)-1$. Similarly, applying~(A2) to $z'+\onebf^a>y$, we obtain $C(y)(a')\ge x(a')+1$ and $C(y)(c)\ge x(c')-1$. Also $C(y)(e)\ge x(e)$ for all $e\in E_w-\{a,a',c,c'\}$. The obtained inequalities together with $C(y)\le y$ imply the equality $C(y)=y$. Now $C(y)=y=z+\onebf^{a'}-\onebf^{c'}=z'+\onebf^a-\onebf^c$ gives~\refeq{zzp_acc}.

To show~(ii), let $d\in U^+_F(x)-\{a\}$. One may assume that either $d\ne a'$, or $d=a'$ and $x(d)\le b(d)-2$ (for the case with $d=a'$ and $x(d)=b(d)-1$ is trivial). Since $(c,a)$ is essential under $x$, we have $C(z+\onebf^{a'})=C(z+\onebf^d)=z$. Then $C(z+\onebf^{a'}+\onebf^d)=z$. (This follows from~\refeq{e_noninterest} if $d=a$. And if $d\ne a'$, it is shown by use of (A2) and (A3).)

Suppose, for a contradiction that $d$ is interesting under $y$. Then $C(y+\onebf^d)$ is equal to either (a) $y+\onebf^d$, or (b) $y+\onebf^d-\onebf^{d'}$ for some $d'\ne d$.

In case~(a), applying (A3) to $z+\onebf^{a'}+\onebf^d> y+\onebf^d$, we would have
   $$
   |x_w|=|z|=|C(z+\onebf^{a'}+\onebf^d)|\ge |C(y+\onebf^d)|=|y+\onebf^d|=|x_w|+1,
   $$
which is impossible.

Now consider case~(b). Applying~(A2) to the inequality $z+\onebf^{a'}+\onebf^d>y+\onebf^d$ and using the equalities $C(z+\onebf^{a'}+\onebf^d)=z$ and $y=z+\onebf^{a'}-\onebf^{c'}$, we obtain
   $$
   z\wedge (z+\onebf^{a'}-\onebf^{c'}+\onebf^d)\le z+\onebf^{a'}-\onebf^{c'}+\onebf^d-\onebf^{d'}.
   $$

In view of $d'\ne d$, this inequality is possible only if $d'=a'$. It follows that $d\ne a'$, and now we can argue as above by considering the vector $z'$ (along with $d\in U^+_F(x)-\{a'\}$) in place of $z$. But then, by symmetry, the edge $d'$ (which is the same as before) must be equal to $a$ rather than $a'$; a contradiction. So~(ii) is valid, as required.

Finally, (iii) easily follows from (i),(ii) and $y=z+\onebf^{a'}-\onebf^{c'}=z'+\onebf^a-\onebf^c$.  \end{proof}

Using an induction, one can obtain from Lemma~\ref{lm:partW} the following
  \begin{corollary} \label{cor:essent_pairs}
For $w\in W$, let $(c(1),a(1)),\ldots,(c(k),a(k))$ be pairwise edge disjoint essential $w$-pairs under $x$. Let $I\subseteq[k]$. Then: {\rm(i)} the vector $z^I:=x_w+\sum_{i\in I}(\onebf^{a(i)}-\onebf^{c(i)})$ is acceptable, i.e. $C_w(z^I)=z^I$; {\rm(ii)} no edge $e\in U^+_F$ is interesting for $w$ under $z^I$; and {\rm(iii)} for each $j\in[k]-I$, the pair $(c(j),a(j))$ is essential under $z^I$. \hfill$\qed$
  \end{corollary}

\noindent\textbf{Proof of Proposition~\ref{pr:xxp}.} For $v\in V$, we write $R^+_v$ for $R^+\cap E_v$, and $R^-_v$ for $R^-\cap E_v$. First of all we observe that $x'_v$ is acceptable for all $v\in V$. This is immediate if $v\notin V_R$. For $f\in F\cap V_R$, this follows from~\refeq{zxab} with $R^+_f=\{a(1),\ldots,a(k)\}$ and $R^-_f=\{c(1),\ldots,c(k)\}$. And for $w\in W\cap V_R$, this follows from Corollary~\ref{cor:essent_pairs} with $R^+_w=\{a(1),\ldots,a(k)\}$, $R^-_w=\{c(1),\ldots,c(k)\}$ and $I=[k]$.

Thus, $x'\in\Ascr$. Applying~\refeq{zxab} to all $f\in F\cap V_R$, we obtain $x'\succ_F x$.

Arguing from contradiction, suppose that $x'$ is not stable and consider a
blocking edge $e=wf$ for $x'$, i.e. $e$ is interesting for $w$ under $x'_w$,
and for $f$ under $x'_f$. Note that $x(e)\le x'(e)$. (For $x(e)>x'(e)$
would imply $e\in R^-$. Then Lemma~\ref{lm:acac} applied to
$R^+_f$ and $R':=R^-_f-\{e\}$ should give $C_f(x'_f+\onebf^{e})=C_f(x_f+ \chi^{R^+_f}-\chi^{R'})=x'_f$, contrary to the fact that $e$ is interesting for $f$ under $x'_f$.) Note also that $x(e)<x'(e)$ would imply $e\in R^+$. But Corollary~\ref{cor:essent_pairs} applied to the set of essential $w$-pairs $(c,a)$ formed by edges $c\in R^-_w$ and $a\in R^+_w$  implies that $e$ is not interesting under $x'_w=x_w-\chi^{R^-_w}+\chi^{R^+_w}$; a contradiction. Therefore, $x(e)=x'(e)$.

Next we assert that the chosen edge $e$ is interesting for $f$ under $x$. For suppose not. Then $C(x_f+\onebf^{e})=x_f$, letting $C:=C_f$, and we obtain (using~\refeq{plott} and~\refeq{zxab}):
 $$
 C(x'_f\vee (x_f+\onebf^{e}))=C(x'_f\vee C(x_f+\onebf^{e}))=C(x'_f\vee x_f)=x'_f.
 $$
On the other hand, for $z':=C(x'_f+\onebf^{e})$ and $p:=x(e)=x'(e)$, we have
  \begin{multline*}
 C(x'_f\vee (x_f+\onebf^{e}))= C(x'_f\vee x_f\vee (p+1)\onebf^{e}) \\
 =C(C(x'_f\vee x_f)\vee(p+1)\onebf^{e})= C(x'_f+\onebf^{e})=z'.
  \end{multline*}
Therefore, $z'=x'_f$. But then $e$ is not interesting for $f$ under $x'$; a
contradiction. Thus, $e$ is interesting for $f$ under $x$.

Finally, since $e$ is interesting for $f$ under $x$, it is not interesting for $w$ under $x$ (by the stability of $x$). Applying Corollary~\ref{cor:essent_pairs} to the set of essential $w$-pairs occurring in $R$ and to the edge $e$ (which is in $U^+_F(x)$), we conclude that $e$ is not interesting for $w$ under $x'_w=x_w+\chi^{R^+_w}-\chi^{R^-_w}$, contradicting the supposition that $e$ is blocking for $x'$. 

Thus, $x'$ admits no blocking edges and therefore is stable. \hfill$\qed$
 \medskip

\noindent\textbf{Proof of Proposition~\ref{pr:xpy}.}
We will use some auxiliary claims. Note that, by~\refeq{AG}(c), for each vertex $v\in V$, the sizes of the restrictions of $x$ and $y$ to $E_v$ are equal: $|x_v|=|y_v|$. Since $y\ne x$, there exists $f\in F$ such that $y_f\ne x_f$; then $x_f\prec_f y_f$. For $v\in V$, define
  $$
  Y^>_v:=\{e\in E_v\colon y(e)>x(e)\}\quad\mbox{and}\quad 
                Y^<_v:=\{e\in E_v\colon y(e)<x(e)\}.
 $$
Clearly if $y_f\ne x_f$, then both $Y^>_f$ and $Y^<_f$ are nonempty. 
\medskip

\noindent\textbf{Claim 1.} \emph{Let $f\in F$ and $y_f\ne x_f$. Let $a=wf\in E_f$ be interesting for $f$ under $x$ (viz. $a\in U^+_F(x)$). Suppose that either {\rm(I):} $a\in Y^>_f$, or {\rm(II):} $y(a)\le x(a)$ and $a$ is not interesting for $f$ under $y$. Then: {\rm(i)} $C_f(x_f+\onebf^a)=x_f+\onebf^a-\onebf^c$ for some $c\in E_f$; {\rm(ii)} $c$ belongs to $U^-_F(x)$ and $Y^<_f$; and {\rm(iii)} $z:=x_f+\onebf^a-\onebf^c$ satisfies $x_f\prec_f z\preceq_f y_f$.
 }
 \medskip
 
\noindent\textbf{Proof} 
We try to unify our analysis in situations (I) and (II) by putting $\eps:=0$ in~(I), and $\eps:=x(a)-y(a)+1$ in~(II). Then $y(a)\le b(a)$, and $C(y_f+\eps\onebf^a)=y_f$ (cf.~\refeq{e_noninterest} in case~(II)), where we write $C$ for $C_f$, and $\onebf^\bullet$ for $\onebf^\bullet_f$.

Since $a\in U^+_F(x)$, two cases are possible (cf.~Lemma~\ref{lm:e_interest}(iv)): (a) $C(x_f+\onebf^a)= x_f+\onebf^a$, or (b) $C(x_f+\onebf^a)= x_f+\onebf^a-\onebf^c$ for some $c\in E_f-\{a\}$.
 
Suppose we are in case~(a). Then $|C(x_f+\onebf^a)|=|x_f|+1$. But applying~(A3) to $x_f+\onebf^a\le (y_f+\eps\onebf^a)\vee x_f$ (which follows from $y(a)+\eps\ge x(a)+1$), we obtain
  $$
  |C(x_f+\onebf^a)|\le |C((y_f+\eps\onebf^a)\vee x_f)|=|C(C(y_f+\eps\onebf^a)\vee x_f)|=|C(y_f\vee x_f)|=|y_f|=|x_f|,
  $$
using $C(y_f+\eps\onebf^a)=y_f$, and $C(y_f\vee x_f)=y_f$; a contradiction.

Thus, case~(b) takes place: $C_f(x_f+\onebf^a)= x_f+\onebf^a-\onebf^c$ for some $c\in E_f-\{a\}$, as required in~(i). Then $(a,c)$ forms a legal $f$-pair for $x$.

This implies that $c\in U^-_F(x)$ (by Lemma~\ref{lm:acd}). Then $y(c)\le x(c)$ (for $y(c)>x(c)$ would imply $c\in U^+_F(x)$, by Lemma~\ref{lm:e_interest}(i),(ii)). Suppose that $y(c)=x(c)=:p$. Then, using $C(y_f+\eps\onebf^a)=y_f$ and $y(a)+\eps\ge x(a)+1$, we have
 \begin{multline*}
 y_f=C(y_f\vee x_f)=C((y_f+\eps\onebf^a)\vee x_f)=C((y_f-\onebf^c)\vee(x_f+\onebf^a)) \\
 =C((y_f-\onebf^c)\vee C(x_f+\onebf^a)) = C((y_f-\onebf^c)\vee (x_f+\onebf^a-\onebf^c)),
 \end{multline*}
whence $y(c)\le p-1$; a contradiction. Thus, $y(c)<x(c)$, as required in~(ii). 

Finally, the relation $x_f\prec_f z$ in~(iii)  is due to Lemma~\ref{lm:acd}. To see the second relation there, put $q:=x(a)$. Then $y_f+\eps\onebf^a=y_f\vee (q+1)\onebf^a$, and using relations $C(y_f+\eps \onebf^a)=y_f$, ~$C(y_f\vee x_f)=y_f$, ~$C(x_f+\onebf^a)=z$, and~\refeq{plott}, we have
  \begin{multline*}
y_f=C(y_f\vee (q+1)\onebf^a)=C(y_f\vee x_f\vee (q+1)\onebf^a) \\
 =C(y_f\vee C(x_f\vee (q+1)\onebf^a))=C(y_f\vee z),
  \end{multline*}
whence $z\preceq_f y_f$, as required in (iii).  
\hfill$\qed$
  \medskip
  
\noindent\textbf{Claim 2.} \emph{Let $c=hf$ be as in Claim~1 (where $h\in W$). Then:
{\rm(i)} the set $Y^>_h$ is nonempty, and each edge $d=hg\in Y^>_h$ is not interesting for $h$ under $x$; {\rm(ii)} there exists $d'\in Y^>_h$ such that $\hat z:=x_h-\onebf^c+\onebf^{d'}$ is acceptable; and {\rm(iii)} $x_h\succ_h \hat z\succeq_h y_h$ is valid.
}
 \medskip
 
\noindent\textbf{Proof} ~Since $|y_h|=|x_h|$ and $c\in Y^<_h$, we have $Y^>_h\ne\emptyset$. For an edge $d=hg \in Y^>_h$, the relations $x(d)<y(d)$ and $x_g\prec_g y_g$ imply that $d$ is interesting for $g$ under $x$ (by Lemma~\ref{lm:e_interest}). Then $d$ is not interesting for $h$ under $x$ (since $x$ is stable), yielding~(i).

To show~(ii), take $d\in Y^>_h$ and form the vectors $z:=x_h-\onebf^c+\onebf^d$  and $z':=y_h\vee z$ (where $\onebf^\bullet$ means $\onebf^\bullet_h$). Then $(x_h\vee y_h)>z'$ (in view of $x(c)>z(c)\ge y(c)$ and $y(d)\ge x(d)+1=z(d)$). Also $C(x_h\vee y_h)=x_h$ (since $x_h\succ_h y_h$), letting $C:=C_h$. 

Applying (A2) to $(x_h\vee y_h)>z'$, we have $C(x_h\vee y_h)\wedge z'\le C(z')$; equivalently, $x_h\wedge z'\le C(z')$. This implies $C(z')\ge x_h-\onebf^c$ (since $z'\ge x_h-\onebf^c+\onebf^d$). On the other hand,~(A3) applied to $(x_h\vee y_h)>z'$ gives $|x_h|\ge |C(z')|$. Also $z'(c)=x(c)-1$. It follows that $C(z')$ is equal to either (a) $x_h-\onebf^c$, or (b) $x_h-\onebf^c+\onebf^{d'}$ for some $d'\in E_h-\{c\}$. 

Case~(a) is impossible since $z'\ge y_h$ implies $|C(z')|\ge|C(y_h)|=|x_h|$ (by~(A3)). Therefore,  $C(z')=x_h-\onebf^c+\onebf^{d'}=:\hat z$. Then $C(\hat z)=\hat z$. Also $z'(d')>x(d')$. This and $z'(d')=(y_h\vee(x_h-\onebf^c+\onebf^d))(d')$ imply $y(d')>x(d')$ (which is valid for $d'=d$ as well). Then $d'\in Y^>_h$, and~(ii) follows. 

Finally, the first relation in (iii) follows from $C(x_h\vee\hat z)=C(x_h+\onebf^{d'})= x_h$ (since $d'$ belongs to $Y^>_h$, and therefore it is not interesting for $h$ under $x$, by~(i)). And to see the second relation in~(iii), consider $z$ and $z'$ as above; then
   $$
   \hat z=C(z')=C(y_h\vee z)=C(y_h\vee (y_h\vee z))=C (y_h\vee C(z'))=C(y_h\vee \hat z). \qquad \qed
   $$

\noindent\textbf{Claim 3.}
\emph{Let $h$ and $c$ be as in Claim~2. Let $a'=hf'\in U^+_F(x)$ be such that the $h$-pair $(c,a')$ is essential under $x$. Then either $y(a')>x(a')$ (viz. $a'\in Y^>_h$), or $y(a')\le x(a')$ and $a'$ is not interesting for the vertex $f'$ under $y$.}
 \medskip
 
\noindent\textbf{Proof}
Consider the edge $d'\in Y^>_h$ and vector $\hat z=x_h-\onebf^c+\onebf^{d'}$ as in Claim~2. Since $(c,a')$ is essential, $a'$ is interesting for $h$ under $\hat z$, and for $z':=x_h-\onebf^c+\onebf^{a'}$, we have
  \begin{equation} \label{eq:xzpzy}
  x_h\succ_h z'\succeq_h \hat z\succeq_h y_h
  \end{equation}
(in view of Lemma~\ref{lm:essent_pair}, \refeq{U+a}, and (iii) in Claim~2).

If $y(a')>x(a')$, we are done. So suppose that $y(a')\le x(a')$. We assert that $a'$ is interesting for the vertex $h$ under $y$. This follows from Lemma~\ref{lm:e_interest}(i),(ii) if $y(a')<x(a')$ (since $y_h\prec_h x_h$). Now assume that $y(a')=x(a')$. Then applying~(A2) to the inequality $z'\vee y_h\ge y_h+\onebf^{a'}$, we have
      $$
      C(z'\vee y_h)\wedge (y_h+\onebf^{a'})\le C(y_h+\onebf^{a'}),
      $$ 
where  $C:=C_h$. Since $C(z'\vee y_h)=z'$ (in view $z'\succeq_h y_h$, by~\refeq{xzpzy}) and $z'(a')=x(a')+1=y(a')+1$, we can conclude that $C(y_h+\onebf^{a'})(a')>y(a')$. Thus, $C(y_h+\onebf^{a'})\ne y_h$, whence $a'$ is interesting for $h$ under $y$, as required.

This and the stability of $y$ imply that $a'=hf'$ must be not interesting for the vertex $f'$ under $y$, yielding the claim. \hfill$\qed$
 \medskip

Now we are ready to finish the proof of the proposition. We start our construction by choosing $f\in F$ with $y_f\ne x_f$ and fixing $a\in Y^>_f$. Then $a\in U^+_F(x)$, and by Claim~1, $a$ generates $f$-legal pair $(a,c=hf)$. The edge $c$ belongs to $U^-_F(x)\cap Y^<_h$, and by Claims~2 and~3, it generates essential $h$-pair $(c,a'=hf')$. Moreover, according to Claim~3, the edge $a'=hf'$ (which is in $U^+_F(x)$) possesses the property that either $a'\in Y^>_{f'}$, or $y(a')\le x(a')$ and $a'$ is not interesting for the vertex $f'\in F$ under $y$; in both cases, $y_{f'}\ne x_{f'}$ is valid. So we get into the hypotheses of Claim~1 with $a'$  (in place of $a$), and can repeat the procedure starting with $a'$.  

By continuing this process, we obtain an ``infinite'' sequence of edges (determined by the first edge $a$), say, sequence $(a_1,c_1,a_2,c_2, \ldots)$, where for each $i$, $(a_i,c_i)$ forms a legal $F$-pair, and $(c_i,a_{i+1})$ forms an essential $W$-pair for $x$. Moreover, as is shown in the above claims,  these pairs obey relations of the form 
  $$
  x_{f(i)}\prec_{f(i)} z(i) \preceq_{f(i)} y_{f(i)}\quad \mbox{and} \quad
     x_{w(i)}\succ_{w(i)} \hat z(i) \succeq_{w(i)} y_{w(i)},
     $$
where: $f(i)$ is the vertex in $F$ shared by the edges $a_i$ and $c_i$;  $w(i)$ is the vertex in $W$ shared by $c_i$ and $a_{i+1}$;  $z(i):=x_{f(i)}+\onebf^{a_i}_{f(i)}-\onebf^{c_i}_{f(i)}$; and $\hat z(i):=x_{w(i)}-\onebf^{c_i}_{w(i)}+\onebf^{a_{i+1}}_{w(i)}$.

Now we extract from the ``infinite'' path $(w(0),a_1,f(1),c_1,w(1),a_2,f(2),c_2,w(2),\ldots)$ a minimal portion between two copies of the same vertex, say, copies $w(j)$ and $w(k)$ of $w\in W$ (where $j<k$). This gives a cycle $R$ which is nothing else than a rotation applicable to $x$. By Proposition~\ref{pr:xxp}, $R$ determines the stable g-matching $x':=x+\chi^R$.

For $f\in F$, let $I_f$ be the set of copies $f(i)$ of $f$ such that $j<i\le k$; then $\{(a_i,c_i)\colon i\in I(f)\}$ is the set of legal $f$-pairs occurring in $R$ (possibly $I_f=\emptyset$).  By Claim~1, for each $i\in I_f$, the vector $z^i:=x_f+\onebf^{a_i}-\onebf^{c_i}$
satisfies $x_f\prec_f z^i\preceq_f y_f$. Then in the lattice
$(\Ascr_f,\succ_f)$, the join $z_f:=\curlyvee(z^i\colon i\in I_f)$ satisfies $x_f\prec_f  z_f \preceq_f y_f$. Lemma~\ref{lm:acac} implies that
  $$
  \curlyvee(z^i\colon i\in I_f)= C_f(\vee(z^i\colon i\in I_f))=
  \vee_{i\in I_f}(x_f+\onebf^{a(i)}-\onebf^{c(i)}).
  $$
Therefore, $z_f$ coincides with the restriction $x'_f$ of $x'$ to $E_f$,
and we can conclude that the g-matching $x'\in \Sscr_x$ satisfies $x\prec_F
x'\preceq_F y$, as required in the proposition. \hfill$\qed\;\qed$
 \medskip

Taking together Propositions~\ref{pr:xxp} and~\ref{pr:xpy}, we conclude with the following
\begin{corollary} \label{cor:immed_succ}
For each $x\in \Sscr$, the set $\Sscr_x$ of stable g-matchings immediately succeeding $x$ in the lattice $(\Sscr,\prec_F)$ consists exactly of the vectors $x+\chi^R$ formed by rotations $R$ applicable to $x$ (viz. $R\in\Rscr(x)$). Also  $|\Sscr_x|=|\Rscr(x)|\le |E|/4$. 
\end{corollary}


\section{Additional properties of rotations} \label{sec:addit_prop}

In this section we describe more properties of rotations in SGMM; they
will be important to construct the poset of rotations in the next section.

Consider a rotation $R\in\Rscr(x)$ for a stable g-matching $x$. In the
previous section we described the transformation of $x$ that consists in
increasing the values of $x$ by 1 on the set $R^+$, and decreasing by 1 on
$R^-$. The resulting g-matching $x':=x+\chi^R$ is again stable
and more preferred for $F$: $x'\succ_F x$; this $x'$ is said to be obtained by shifting $x$ with weight 1 along $R$. We, however,
can try to increase this weight.
\medskip

\noindent\textbf{Definition 6.} Let $x\in\Sscr$ and $R\in\Rscr(x)$. A weight $\lambda\in\Zset_{>0}$ is called \emph{feasible} for $R$ under $x$ if shifting along $R$ with weight 1 can be repeated, step by step, $\lambda$ times, i.e. the sequence $x=x_0,x_1,\ldots, x_\lambda$ defined by $x_i:=x_{i-1}+\chi^{R}$, $i=1,\ldots,\lambda$, consists of stable g-matchings. In particular, $\lambda$ is at most $\min\{\min\{b(e)-x(e)\colon e\in R^+\},\min\{x(e)\colon e\in R^-\}$. We say that $x_\lambda$ is obtained from $x$ by \emph{shifting with weight $\lambda$ along} $R$, or \emph{by applying $R$ with weight $\lambda$}. The maximal feasible $\lambda$ for $(x,R)$ is denoted by $\tau_R(x)$. 
\medskip

Using Lemmas~\ref{lm:acac} and~\ref{lm:partW}, one
can realize that for $\lambda\in\Zset_{>0}$,
   \begin{numitem1} \label{eq:l<t}
if $\lambda<\tau_R(x)$, then shifting $x$ with weight $\lambda$ along $R$ preserves
the active graph, i.e., $\Gamma(x')=\Gamma(x)$ holds for
$x':=x+\lambda\chi^R$, implying $\Rscr(x')=\Rscr(x)$;
also $\tau_R(x')=\tau_R(x)-\lambda$ and $\tau_{R'}(x')=\tau_{R'}(x)$ for all
$R'\in\Rscr(x)-\{R\}$.
  \end{numitem1}

On the other hand,
  \begin{numitem1} \label{eq:l=t}
under shifting $x$ with weight $\tau_R(x)$ along $R\in\Rscr(x)$, yielding $x'$, at least one of the following four
\emph{events} happens: (I) some $e\in R^-$ becomes 0-valued: $x'(e)=0$; or (II)
some $e\in R^+$ becomes saturated: $x'(e)=b(e)$; or (III) there is $f\in F_R$
such that some legal $f$-pair $(a,c)$ occurring in $R$ is destroyed (i.e. $C_f(z+\onebf^{a})$ is different from $z+\onebf^{a}-\onebf^{c}$, where $z:=x_f+\tau_R(x) \chi{^R}\rest{E_f}$ and $z(a)<b(a)$); or (IV) some essential $W$-pair $(c,a)$ occurring in $R$ is destroyed.
  \end{numitem1}
  
Properties~\refeq{l<t} and~\refeq{l=t} imply that rotations with feasible
weights for $x$ commute. More precisely, using Lemmas~\ref{lm:acac} and~\ref{lm:partW}, one can obtain the following
\begin{corollary} \label{cor:commute}
Let $\Rscr'\subseteq\Rscr(x)$ and let $\lambda:\Rscr'\to \Zset_+$ be such
that $\lambda(R)\le \tau_R(x)$ for each $R\in\Rscr'$. Then the g-matching
$x':=x+\sum(\lambda(R) \chi^R \colon R\in\Rscr')$ is stable,
each $R\in \Rscr'$ with $\lambda(R)<\tau_R(x)$ is an applicable rotation for $x'$ having the maximal feasible weight $\tau_R(x')=\tau_R(x)-\lambda(R)$, and each $R'\in \Rscr(x)-\Rscr'$ is an applicable rotation for $x'$ with $\tau_{R'}(x')=\tau_{R'}(x)$. In particular, rotations in $\Rscr(x)$ can be applied in any order.
\end{corollary}

One more notion will be of importance for us. This is close to the definition of a maximal (or connected, or dense) chain in a lattice.
 \medskip

\noindent\textbf{Definition 7.} Let $\Tscr$ be a sequence $x_0,x_1,\ldots,x_N$
of stable g-matchings such that each $x_{i}$ ($1\le i\le N$) is obtained from $x_{i-1}$ by shifting with a feasible weight $\lambda_i>0$ along a rotation $R_i\in\Rscr(x_{i-1})$. In particular, $x_0\prec_F\cdots\prec_F x_N$. We call $\Tscr$ a \emph{route} from $x_1$ to $x_N$, and liberally say that a rotation $R_i$ as above \emph{is used}, or \emph{occurs}, in $\Tscr$. This route is called \emph{non-excessive} if $i<j$ and $R_i=R_j$ imply $\lambda_i=\tau_{R_i}(x_{i-1})$. When $\lambda_i=\tau_{R_i}(x_{i-1})$ for all $i$, we say that $\Tscr$ is \emph{principal}. A principal route from $\xmin$ to $\xmax$ is called \emph{full}. 
 \medskip

Using Corollary~\ref{cor:commute}, one can see that
  \begin{numitem1} \label{eq:non-excess}
for any $x,y\in\Sscr$ with $x\prec_F y$, there exists a non-excessive route from $x$ to $y$.
  \end{numitem1}

\noindent\textbf{Remark~1.} A reasonable question is how large can be  the length $N$ of a non-excessive route $\Tscr$. It is known (see e.g.~\cite{DM}) that in case of stable allocation problem each rotation (as a cycle in $G$) can be used in a non-excessive $\Tscr$ at most once, and the number of rotations does not exceed $|E|/4$; so $|\Tscr|$ is $O(|E|)$. However, already for the generalized allocation problem with integer data considered in~\cite{karz3} (where the side $W$ is endowed with linear orders, while $F$ with general CFs), the same rotation may be used in a non-excessive $\Tscr$ more than once, possibly many times; as a consequence, the length $|\Tscr|$ can be ``large''. Such a behavior is demonstrated by an example in~\cite[Appendix]{karz3} in which the basic graph $G$ has six vertices, there are only two rotations at all, and these rotations are intermixed in a non-excessive route $\Tscr$ of length $|\Tscr|=\qmax+1$, where $\qmax$ is the maximum quota on $W$ (which can be arbitrarily large). 
 \medskip 
 
Return to a general case of SGMM. For a non-excessive route $\Tscr$, let $\Rscr(\Tscr)$ denote the set of \emph{different} rotations used in $\Tscr$. For $R\in\Rscr(\Tscr)$, let $\Pi_R(\Tscr)$ denote the family, with possible replications, of the pairs $(R,\lambda)$ (\emph{weighted rotations}) used in $\Tscr$. (The same pair $(R,\lambda)$ may be repeated in $\Pi_R(\Tscr)$ many times.) Define $\Pi(\Tscr):=\cup(\Pi_R(\Tscr)\colon R\in \Rscr(\Tscr))$.

Extending well-known nice properties of rotations, one can show that the family $\Pi(\Tscr)$ is the same for all full routes $\Tscr$. (Initially an invariance property of this sort was revealed by Irving and Leather~\cite{IL} for (unweighted) rotations in the classical stable marriages problem and subsequently was  demonstrated in the literature for more general models of stability. This is analogous, in a sense, to the fundamental fact in Birkhoff~\cite{birk} that the set of prime ideals constructed by use of a maximal chain in a finite distributive lattice does not depend on the chain; it will be reviewed in the next section.) It is convenient to us to extend such an invariance property to arbitrary pairs of comparable stable g-matchings in our model, not necessarily $(\xmin,\xmax)$.

  \begin{prop} \label{pr:invar_rot}
Let $x,y\in \Sscr$ and $x\prec_F y$. Then for all non-excessive routes $\Tscr$ going from $x$ to $y$, the family $\Pi(\Tscr)$ is the same. A similar property is valid relative to principal routes as well (when $y$ is reachable from $x$ by a principal route).
 \end{prop}
 \begin{proof}
Let $\Xscr$ denote the set of stable g-matchings $x'$ such that there is a non-excessive route from $x$ to $y$ passing $x'$ (in particular, $x\preceq_F x'\preceq_F y$). Let us say that $x'\in\Xscr$ is \emph{bad} if there exist two non-excessive routes $\Tscr,\Tscr'$ from $x'$ to $y$ such that $\Pi(\Tscr)\ne \Pi(\Tscr')$, and \emph{good} otherwise. One has to show that $x$ is good (when $y$ is fixed).

Suppose this is not so, and consider a bad g-matching $x'\in\Xscr$ of maximal height, in the sense that each g-matching $z\in\Xscr$ with $x'\prec_F z\preceq_F
y$ is already good. In any non-excessive route from $x'$ to $y$, the first
g-matching  after $x'$ is obtained from $x'$ by applying a certain weighted
rotation from $\Rscr(x')$. 
Since $x'$ is bad, there are two non-excessive routes $\Tscr,\Tscr'$ from $x'$ to $y$ such that $\Pi(\Tscr)\ne\Pi(\Tscr')$. Let $z$ ($z'$) be the first g-matching in $\Tscr$ (resp. $\Tscr'$) after $x'$. Then $z$ is obtained from $x'$ by shifting with some weight $\lambda$ along a rotation $R\in\Rscr(x')$, and similarly, $z'$ is obtained from $x'$ by shifting with some weight $\lambda'$ along a rotation $R'\in\Rscr(x')$. We have $z,z'\in \Xscr$ and $R\ne R'$. The maximal choice of $x'$ implies that both $z$ and $z'$ are good.

At the same time, the pairs $(R,\lambda)$ and $(R',\lambda')$ commute at $x'$
(cf. Corollary~\ref{cor:commute}), i.e. $R$ taken with weight $\lambda$ is applicable to $z'$, and $R'$ taken with weight $\lambda'$ is applicable to $z$. Let $z''$ be obtained from $x'$ by applying $(R,\lambda)$ followed by applying $(R',\lambda')$ (or vice versa). Then $z''\prec_F y$. Let $\Tscr''$ be a non-excessive route from $z''$ to $y$. 

Now form the route $\tilde \Tscr$ ($\tilde\Tscr'$) from $x'$ to $y$ that begins with $x',z,z''$ (resp. $x',z',z''$) and then continues as $\Tscr''$. One can see that both routes $\tilde \Tscr$ and $\tilde\Tscr'$ are non-excessive. (Indeed, it suffices to check the weights of $R$ and $R'$. If $\lambda=\tau_R(x')$ and $\lambda'=\tau_{R'}(x')$, we are done. And if, say, $\lambda<\tau_R(x')$, then $R$ cannot be used in the part of $\Tscr$ after $z$ (since $\Tscr$ is non-excessive), and therefore it cannot be used in $\Tscr''$ (since $z$ is good).) We have $\Pi(\tilde\Tscr)=\Pi(\tilde\Tscr')$. Since both $z,z'$ are good, there must be
$\Pi(\tilde\Tscr)=\Pi(\Tscr)$ and $\Pi(\tilde\Tscr')=\Pi(\Tscr')$. But then
$\Pi(\Tscr)=\Pi(\Tscr')$; a contradiction.

For principal routes, the argument is similar (even simpler).
 \end{proof}

It follows that for all full routes $\Tscr$, the family $\Pi(\Tscr)$  of weighted rotations (with possible replications) is the same; we denote it by $\Pi=\Pi_{G,b,C}$ (so all possible rotations applicable to stable g-matchings occur in $\Pi$). Also we write $\Rscr$ for the set of different (unweighted) rotations $\Rscr(\Tscr)$, and for a fixed $R\in\Rscr$, write $\Pi_R$ for the family $\Pi_R(\Tscr)$ of pairs $(R,\lambda)$ used in $\Tscr$.
  \medskip
  
\noindent\textbf{Definition 8.} A stable g-matching $x\in\Sscr$ is called \emph{principal} if there is a principal route from $\xmin$ to $x$. The set of principal g-matchings is denoted by $\Sscr^\ast$, and the restriction of $\prec_F$ to $\Sscr^\ast$ is denoted by $\prec_F^\ast$, or briefly by $\prec^\ast$.
One can see that 
\begin{numitem1} \label{eq:S_ast}
$(\Sscr^\ast,\prec^\ast)$ forms a distributive sublattice of $(\Sscr,\prec)$.
 \end{numitem1}

In the next section we explain how to arrange a partial order on $\Pi$ so as to obtain the desired representations for the lattices $(\Sscr^\ast,\prec^\ast)$ and $(\Sscr,\prec)$.


\section{Poset of rotations} \label{sec:poset_rot}

Our construction of poset representation for $(\Sscr,\prec_F)$ relies on a somewhat simpler construction of a representation for the principal sublattice $(\Sscr^\ast,\prec^\ast=\prec_F^\ast)$.

Let $x\in\Sscr^\ast$. By Proposition~\ref{pr:invar_rot}, for any $y\in \Sscr^\ast$ with $y\succ x$, the family $\Pi(\Tscr)$ of pairs (weighted rotations) $(R,\tau_R)$, with possible replications, is the same for all principal routes $\Tscr$ from $x$ to $y$. We denote it by $\Pi^{x,y}$. Also we denote the set $\Rscr(\Tscr)$ of \emph{different} rotations used in $\Tscr$ by $\Rscr^{x,y}$, and for $R\in\Rscr^{x,y}$, we write $\Pi_R^{x,y}$ for $\Pi_R(\Tscr)$. When $x=\xmin$, we may abbreviate such notation as $\Pi^y$, $\Rscr^y$, $\Pi_R^y$, respectively. When, in addition, $y=\xmax$, we use further abbreviations $\Pi,\Rscr,\Pi_R$.

Let us associate with $x$ the numerical function $\omega^x:\Rscr\to \Zset_+$ taking values $\omega^x(R):=|\Pi_R^x|$ for $R\in\Rscr^x$, and 0 for $R\in\Rscr-\Rscr^x$. A nice property is that $\omega^x$ determines $x$.

\begin{lemma} \label{lm:omega_x}
Let $x,x'\in\Sscr^\ast$ and let $\omega^x=\omega^{x'}$. Then $x=x'$.
  \end{lemma}
\begin{proof} 
~Suppose, for a contradiction, that $x\ne x'$. Let $y:=x\curlywedge^{\!\ast} x'$ (the greatest lower bound for $x,x'$ in $(\Sscr^\ast,\prec^\ast$)). Consider a principal route $\Tscr=(y=x_0,x_1,\ldots, x_N)$ from $y$ to $x$ and a principal route $\Tscr'=(y,x'_1,\ldots, x'_{N})$ from $y$ to $x'$; then $\omega^x=\omega^{x'}$ implies $N=N'$ and $\Rscr^{y,x}=\Rscr^{y,x'}$. Let $(R_1,\ldots,R_N)$ and $(R'_1,\ldots,R'_N)$ be the sequences of rotations used in $\Tscr$ and $\Tscr'$, respectively. Then $R_1$ and $R'_1$ are different rotations commuting at $y$ (cf. Corollary~\ref{cor:commute}). So we can apply $(R'_1,\tau_{R'_1})$ to $x_1$ (or, equivalently, $(R_1,\tau_{R_1})$ to $x'_1$), obtaining a principal g-matching; denote it by $z_2$. The case $z_2\in\Tscr$ (equivalently, $R'_1=R_2$) would imply $y\prec x'_1\prec x,x'$, contradicting the definition of $y$. So $z_2\notin\Tscr$, and therefore, $R_2$ and $R'_1$ are different rotations applicable to $x_1$.

Using similar reasonings for $x_1,R_2,R'_1$, we then observe that the principal g-matching  obtained by applying  $(R'_1,\tau_{R'_1})$ to $x_2$ (or, equivalently, $(R_2,\tau_{R_2})$ to $z_2$) cannot belong to $\Tscr$ (or, equivalently, $R'_1\ne R_3$), for otherwise we would have $y\prec x'_1\prec x,x'$ again. Continuing this process, we sooner or later come to the situation when $R'_1=R_i$ for some $i\le N$, thus coming to a contradiction with the definition of $y$. 
\end{proof}
  
Arguing in a similar spirit, one can conclude that
  \begin{numitem1} \label{eq:omega-omega}
  any $x,x'\in\Sscr^\ast$ satisfy $\omega^{\,x\curlywedge^{\!\ast} x'}=\omega^{x}\wedge \omega^{x'}$ and $\omega^{\,x\curlyvee^{\ast} x'}=\omega^x\vee\omega^{x'}$,
  \end{numitem1}
where $\curlywedge^{\!\ast}$ and $\curlyvee^{\!\ast}$ mean the meet and join  in $(\Sscr^\ast,\prec^\ast)$. In other words, $\omega$ determines an isomorphism between the principal lattice $(\Sscr^\ast,\prec^\ast)$ and the vector sublattice $(\omega^{\Sscr^\ast},<)$ in $\Zset^\Rscr_+$. We will use this correspondence in parallel with Birkhoff's constructions.
 
Next we give a review of constructions and facts from~\cite{birk} that are important to us.


\subsection{Birkhoff's constructions} \label{ssec:birk}

Let $(\Lscr,\prec)$ be a finite distributive lattice with join $\curlyvee$, meet $\curlywedge$, minimal element $O$ and maximal element $I$. A set $\Pscr\subseteq \Lscr$ is called an \emph{ideal} of $(\Lscr,\prec)$ if
   $$
   X\curlyvee Y\in\Pscr\;\; \mbox{and}\;\; X\curlywedge Z\in\Pscr\;\;\mbox{for any}\;\; X,Y\in\Pscr\;\;\mbox{and}\;\; Z\in \Lscr,
   $$
and a set $\Dscr\subseteq \Lscr$ is called a \emph{dual ideal}, or a \emph{filter}, if
   $$
   X\curlywedge Y\in\Dscr\;\; \mbox{and}\;\; X\curlyvee Z\in\Dscr\;\;\mbox{for any}\;\; X,Y\in\Dscr\;\;\mbox{and}\;\; Z\in \Lscr.
   $$
Then any ideal $\Pscr$ (filter $\Dscr$) is determined by an element $X\in\Lscr$ and consists of all $Y\in\Lscr$ such that $Y\preceq X$ (resp. $Y\succeq X$); we denote this $X$ as ${\rm MAX}(\Pscr)$ (resp. ${\rm MIN}(\Dscr)$).

An ideal $\Pscr$ is called \emph{prime} if its complement $\Lscr-\Pscr$ is a filter. Following~\cite{birk}, these especial ideals can be extracted by handling an arbitrary maximal (dense) chain in $(\Lscr,\prec)$, i.e. a sequence $\Cscr=(X_0,X_1,\ldots, X_n)$ in $\Lscr$ such that $X_0=O$, $X_n=I$, and each $X_i$ ($i>0$) is an immediate successor of $X_{i-1}$. More precisely, consider $i\in[n]$. The distributivity of $\Lscr$ and the absence of elements of $\Lscr$ between $X_{i-1}$ and $X_i$ imply that
  \begin{numitem1} \label{eq:XiYXi}
for any $Y\in\Lscr$, the element $X_{i-1}\curlyvee(Y\curlywedge X_i)$ coincides with $(X_{i-1}\curlyvee Y)\curlywedge X_i$ and is equal to either $X_{i-1}$ or $X_i$.
 \end{numitem1}
 
Let $\Pscr_i:=\{Y\in\Lscr\colon (X_{i-1}\curlyvee Y)\curlywedge X_i=X_{i-1}\}$ and $\Dscr_i:=\{Y\in\Lscr\colon (X_{i-1}\curlyvee Y)\curlywedge X_i=X_i\}$. Using~\refeq{XiYXi}, we can see that $P_i$ is an ideal and $\Dscr_i$ is a filter. Also $\Dscr_i$ is complementary to $\Pscr_i$. So $\Pscr_i$ is a prime ideal. The following fact is of importance.
  \begin{prop} \label{pr:Pscr_i} {\rm(\cite[Th.~4]{birk})}
~$\Pscr_1,\ldots,\Pscr_n$ are exactly the prime ideals of $(\Lscr,\prec)$.
  \end{prop}

So $(\Lscr,\prec)$ has as many prime ideals as the number $n$ of ``links'' in a maximal chain $\Cscr$, and the set of these can be constructed as above by taking such a $\Cscr$ arbitrarily.  
 
Next, following terminology of~\cite{birk}, by a \emph{prime factor} of the lattice $(\Lscr,\prec)$ one means any ``symbol'' $Y/X$, where $X,Y\in\Lscr$ and $X$ immediately precedes $Y$; for brevity, we will call $Y/X$ a \emph{p-factor}. Clearly for each p-factor $Y/X$, there is a maximal chain $(X_0,X_1,\ldots,X_n)$ such that $X=X_{i-1}$ and $Y=X_i$ for some $i\in[n]$. So there is a prime ideal $\Pscr$ such that $X\in \Pscr$ and $Y\in\Lscr-\Pscr$, and this $\Pscr$ is determined uniquely.

It follows that the set of all p-factors can be partitioned into $n$ groups, where each group is formed by the p-factors $Y/X$ related to the same prime ideal $\Pscr$; we call it the \emph{girdle} associated with $\Pscr$ and denote it as $\Gscr=\Gscr(\Pscr)$. One can see that $\Gscr(\Pscr)$ is exactly the set of p-factors $Y/X$ such that $X\in\Pscr$ and $Y\in\Lscr-\Pscr$. Note also that the p-factors in a girdle are pairwise disjoint, i.e.
  \begin{numitem1} \label{eq:distinct_cleav}
if $Y/X$ and $Y'/X'$ are distinct p-factors in $\Gscr(\Pscr)$, then $X\ne X'$ and $Y\ne Y'$.
  \end{numitem1}
Indeed, suppose that $X=X'$ and $Y\ne Y'$. Then $X\in\Pscr$ and $Y,Y'\in\Lscr-\Pscr$. By~\refeq{XiYXi} applied to $Y$ and $Y'/X$, we have $(X\curlyvee Y)\curlywedge Y'=Y'$. But $X\curlyvee Y=Y$ and $Y\curlywedge Y'\ne Y'$; a contradiction. The case $X\ne X'$ and $Y=Y'$ is examined similarly.

Now we arrange the useful partial order $\lessdot$ on the set $\frakG$ of girdles by setting $\Gscr\lessdot\Gscr'$ if for each maximal chain $\Cscr$, a representative (p-factor) of $\Gscr$ occurs in $\Cscr$ earlier than that of $\Gscr'$. For a girdle $\Gscr=\Gscr(\Pscr)$, let $\mu(\Gscr)$ denote the maximal element ${\rm MAX}(\Pscr)$ of $\Pscr$. The next assertion will be  important for our further purposes.
  \begin{lemma} \label{lm:two_girdles}
Let $\Gscr,\Gscr'\in\frakG$. Then $\Gscr\lessdot\Gscr'$ if and only if $\mu(\Gscr)\prec \mu(\Gscr')$.
 \end{lemma}
  \begin{proof}
~Let $\Gscr=\Gscr(\Pscr)$ and $\Gscr'=\Gscr(\Pscr')$ be incomparable. Then there are maximal chains $\Cscr=(X_0,X_1,\ldots,X_n)$ and $\Cscr'=(Y_0,Y_1,\ldots,Y_n)$ that meet representatives of $\Gscr$ and $\Gscr'$ in different orders, say, $\Gscr$ contains $X_i/X_{i-1}$ and $Y_j/Y_{j-1}$, $\Gscr'$ contains $X_k/X_{k-1}$ and $Y_\ell/Y_{\ell-1}$, where $i<k$ and $\ell<j$. 

Suppose that the maximal elements of $\Pscr$ and $\Pscr'$ are comparable, say, $\mu(\Gscr)\prec \mu(\Gscr')$. We have $Y_\ell\preceq Y_{j-1}$ (since $\ell<j$) and $Y_{j-1}\preceq\mu(\Gscr)$ (since $Y_{j-1}\in\Pscr$ and $\mu(\Gscr)={\rm MAX}(\Pscr)$). This implies $Y_\ell\prec \mu(\Gscr')$. But $\mu(\Gscr')$ belongs to $\Pscr'$, while $Y_\ell$ does to the filter $\Lscr-\Pscr'$; a contradiction. Thus, $\mu(\Gscr)$ and $\mu(\Gscr')$ are incomparable, as required.

Now let $\Gscr$ and $\Gscr'$ be comparable, say, $\Gscr\lessdot \Gscr'$. Then each maximal chain $\Cscr$ meets $\Gscr$ earlier than $\Gscr'$. Considering $\Cscr$ that contains the element $\mu(\Gscr)$ and taking the member $Y/X$ of $\Gscr'$ occurring in $\Cscr$, we have $\mu(G)\prec X\preceq \mu(\Gscr')$, as required.  
  \end{proof} 
  
One more useful property is as follows:
 \begin{numitem1} \label{eq:MIN-MAX}
for prime ideals $\Pscr,\Pscr'$ and the filters $\Dscr:=\Lscr-\Pscr$ and $\Dscr':=\Lscr-\Pscr'$, if ${\rm MAX}(\Pscr)\prec {\rm MAX}(\Pscr')$, then ${\rm MIN}(\Dscr)\prec {\rm MIN}(\Dscr')$, and vice versa.
 \end{numitem1}
Indeed, ${\rm MAX}(\Pscr)\prec {\rm MAX}(\Pscr')$ implies that $\Pscr\subset \Pscr'$. Then $\Dscr\supset \Dscr'$, whence ${\rm MIN}(\Dscr)\prec {\rm MIN}(\Dscr')$. The converse is easy as well.
  \smallskip
  
For a girdle $\Gscr=\Gscr(\Pscr)$, we will write $\nu(\Gscr)$ for ${\rm MIN}(\Lscr-\Pscr)$. Lemma~\ref{lm:two_girdles} and~\refeq{MIN-MAX} imply the following
  \begin{corollary} \label{cor:agree}
~$\Gscr,\Gscr'\in\frakG$ satisfy $\Gscr\lessdot \Gscr'$ if and only if $\nu(\Gscr)\prec \nu(\Gscr')$.
 \end{corollary}
 
The above properties enable us to obtain a poset representation for a finite distributive lattice $(\Lscr,\prec)$ that we will essentially use for our stability model (this representation is somewhat different from the one in assertion $(9\alpha)$ of~\cite{birk}). It is established in terms of closed subsets in the poset $(\frakG,\lessdot)$, where a subset $\Ascr\subseteq \frakG$ is called \emph{closed} if $\Gscr\in \Ascr$, $\Gscr'\in\frakG$ and $\Gscr'\lessdot \Gscr$ imply $\Gscr'\in\Ascr$. 

For an element $X\in\Lscr$ and a maximal chain $\Cscr$ containing $X$, let $\sigma(X)$ denote the set of girdles that meet $\Cscr$ before $X$ (including the p-factor of the form $X/X'$). (This $\sigma(X)$ does not depend on $\Cscr$; for one can see that a girdle $\Gscr$ meets the part of $\Cscr$ between $O$ and $X$ if and only if it does so for any other maximal chain passing $X$.)

The collection $\frakC $ of closed sets in $(\frakG,\lessdot)$ preserves under taking intersections and unions, forming a distributive lattice (a ``ring of sets'' in $\frakG$) with $\emptyset$ and $\frakG$ as the minimal and maximal element, respectively. We can show the following (cf.~\cite[$(9\alpha)$]{birk})
 \begin{prop} \label{pr:poset_repr}
~The correspondence $X\mapsto\sigma(X)$ gives a bijection between $\Lscr$ and $\frakC$. Furthermore, $X\prec Y$ implies $\sigma(X)\subset \sigma(Y)$, and vice versa (i.e. $\sigma$ establishes an isomorphism between the lattices $(\Lscr,\prec)$ and $(\frakC,\subset)$).
  \end{prop}
\begin{proof}
~One direction is easy: for each $X\in\Lscr$, the set $\sigma(X)$ is closed. Indeed, take a maximal chain $\Cscr$ passing $X$. If a girdle $\Gscr$ meets the part of $\Cscr$ from $O$ to $X$, then any $\Gscr'$ with $\Gscr'\lessdot \Gscr$ does so as well.

To see the other direction, consider a closed set $\Ascr\subseteq \frakG$. For  a girdle $\Gscr=\Gscr(\Pscr)\in\Ascr$, take the minimal element $\nu(\Gscr)$ in the filter $\Lscr-\Pscr$, and consider an arbitrary maximal chain $\Cscr$ passing $\nu(\Gscr)$. Let $(X_0,X_1,\ldots,X_k)$ be the part of $\Cscr$ from $O=X_0$ to $\nu(\Gscr)=X_k$, and for each $i\in[k]$, let $\Gscr_i=\Gscr(\Pscr_i)$ be the girdle containing $X_i/X_{i-1}$. One can see that $\Gscr=\Gscr_k$ and that the girdles $\Gscr'$ satisfying $\Gscr'\lessdot \Gscr$ are exactly $\Gscr_1,\ldots,\Gscr_{k-1}$.

It follows that $\sigma(X_k)=\{\Gscr_1,\ldots,\Gscr_k\}$. Now take the join $X\in\Lscr$ of elements $\nu(\Gscr)$ over $\Gscr\in\Ascr$ (equivalently, over the maximal girdles in $\Ascr$). Consider a grid $\Gscr'\in \sigma(X)$ and let $\Gscr'=\Gscr(\Pscr')$; then $X$ belongs to the filter $\Dscr':=\Lscr-\Pscr'$. Suppose there is $\Gscr\in\Ascr$ such that $\nu(\Gscr)\in\Dscr'$. Then $\Gscr'\,\underline{\lessdot}\,\, \Gscr$ (by reasonings above), and therefore, $\Gscr'\in\Ascr$. Now suppose that $\nu(\Gscr)\notin \Dscr'$ for all $\Gscr\in\Ascr$. Then all elements $\nu(\Gscr)$, $\Gscr\in\Ascr$, belong to $\Pscr'$, whence their join $X$ is in $\Pscr'$ as well; a contradiction. Thus, $\sigma(X)=\Ascr$.
\end{proof}

We finish this subsection with two more useful facts (needed to us later).
  \begin{lemma} \label{lm:cross}
Girdles $\Gscr=\Gscr(\Pscr)$ and $\Gscr'=\Gscr(\Pscr')$ are incomparable if and only if there are $X,Y,Z\in\Lscr$ such that $Y/X\in\Gscr$ and $Z/X\in\Gscr'$.
  \end{lemma}
  \begin{proof}
Let $\Gscr=\Gscr(\Pscr)$ and $\Gscr'=\Gscr(\Pscr')$ be incomparable. Consider the elements $\mu(\Gscr)$ ($={\rm MAX}(\Pscr)$) and $\mu(\Gscr')$ ($={\rm MAX}(\Pscr')$), and let $X$ be their meet $\mu(\Gscr)\curlywedge \mu(\Gscr')$. Take maximal chains $\Cscr,\Cscr'$ such that $\Cscr$ passes $X$ and $\mu(\Gscr)$, and $\Cscr'$ passes $X$ and $\mu(\Gscr')$. The incomparability of $\Gscr$ and $\Gscr'$ (equivalently, of $\mu(\Gscr)$ and $\mu(\Gscr')$, by Lemma~\ref{lm:two_girdles}) implies that $\Gscr$ meets $\Cscr'$ with a p-factor contained in its part from $X$ to $\mu(\Gscr')$, while $\Gscr'$ meets $\Cscr$ with a p-factor contained in its part from $X$ to $\mu(\Gscr)$. 

Moreover, one can see that $\Gscr$ contains $Y/X$ and $\Gscr'$ contains $Z/X$, where $Y$ ($Z$) is the element of $\Cscr'$ (resp. $\Cscr$) next to $X$. (For if, say, $\Gscr$ meets $\Cscr'$ with a p-factor $Y'/X'$ such that $X'\ne X$, then $X\prec X'\prec \mu(\Gscr')$. Also $X'\in\Pscr$ implies $X'\prec\mu(\Gscr)$. Then $X$ cannot be $\mu(\Gscr)\curlywedge\mu(\Gscr')$; a contradiction.) So $X,Y,Z$ are as required. 

Conversely, suppose there are $X,Y,Z$ such that $Y/X\in \Gscr=\Gscr(\Pscr)$ and $Z/X\in\Gscr'=\Gscr(\Pscr')$. Let $\Dscr:=\Lscr-\Pscr$ and $\Dscr':=\Lscr-\Pscr'$. Then $X\in\Pscr\cap\Pscr'$, $Y\in\Dscr\cap \Pscr'$ and $Z\in\Dscr'\cap\Pscr$. It follows that $\mu(\Gscr)\in\Dscr'$ and $\mu(\Gscr')\in\Dscr$. This is possible only if $\mu(\Gscr)$ and $\mu(\Gscr')$ are incomparable (under $\prec$). Then $\Gscr,\Gscr'$ are incomparable as well.
  \end{proof}.
 \begin{lemma} \label{lm:immed_suc}
Let $\Gscr=\Gscr(\Pscr)\in\frakG$, $X:=\mu(\Gscr)$ and $Y/X\in\Gscr$. Let $Z_1,\ldots,Z_k$ be the immediate successors of $Y$ in $(\Lscr,\prec)$. Then the girdles containing the p-factors $Z_1/Y,\ldots, Z_k/Y$ are exactly the immediate successors of $\Gscr$ in $(\frakG,\lessdot)$.
 \end{lemma}
   \begin{proof}
~For $i=1,\ldots,k$, let $\Gscr_k=\Gscr(\Pscr_k)$ be the girdle containing $Z_i/Y$. The cases $k=0,1$ are trivial, so assume that $k\ge 2$. By Lemma~\ref{lm:cross}, the girdles $\Gscr_1,\ldots,\Gscr_k$ are pairwise incomparable. Since each prime ideal $\Pscr_i$ contains the element $Y$ which is greater than $X=\mu(\Gscr)$, all girdles $\Gscr_i$ satisfy $\Gscr\lessdot \Gscr_i$.  To see that they form exactly the set of  immediate successors of $\Gscr$, consider an arbitrary girdle $\Gscr'=\Gscr(\Pscr')$ such that $\Gscr\lessdot \Gscr'$. Then $Y\in\Pscr'$ and $\{Z_1,\ldots,Z_k\}\cap\Pscr'\ne\emptyset$. Two cases are possible: (a) there is $Z_i$ not in $\Pscr'$, and (b) $\{Z_1,\ldots,Z_k\}\subseteq\Pscr'$.

In case~(a), we have $Z_i/Y\in\Gscr'$, whence $\Gscr'=\Gscr_i$. So assume that we are in case~(b). Then $\Gscr'\ne\Gscr_i$ for all $i$. If there is $i\in[k]$ such that $\mu(\Gscr_i)\in \Pscr'$, then $\Gscr_i\lessdot\Gscr'$, whence $\Gscr'$ is not an immediate successor of $\Gscr$. Now suppose that $\mu(\Gscr_i)\notin \Pscr'$ for all $i$. We assert that this is not the case. To show this, take a dense chain $\Cscr_1=(U_0,\ldots,U_p)$ from $U_0=Y$ to $U_p=\mu(\Gscr_1)$, and a dense chain $\Cscr_2=(V_0,\ldots,V_q)$ from $V_0=Y$ to $V_q=\mu(\Gscr_2)$. Since $\mu(\Gscr_i)\notin \Pscr'$ for $i=1,2$, there are $1\le\alpha<p$ and $1\le\beta<q$ such that $U_\alpha,V_\beta\in \Pscr'$ but $U_{\alpha+1},V_{\beta+1}\in \Lscr-\Pscr'=:\Dscr'$. 
Consider the elements 
 \begin{equation} \label{eq:ABCUV}
 A:=U_\alpha\curlyvee V_\beta,\quad B:=A\curlyvee U_{\alpha+1}\quad\mbox{and}\quad C:=A\curlyvee V_{\beta+1}.
 \end{equation}
Then $U_\alpha,V_\beta\in\Pscr'$ implies $A\in\Pscr'$. Also $U_{\alpha+1}\in\Dscr'$ and $U_{\alpha+1}\preceq B$ imply $B\in \Dscr'$. Similarly $C\in \Dscr'$. We can conclude from~\refeq{ABCUV} that  $B,C$ are immediately successors of $A$. Then $B/A,C/A\in\Gscr'$. This implies $B=C$, by~\refeq{distinct_cleav}. Using this, we observe that
  $$
  V_\beta\curlyvee U_{\alpha+1}=V_\beta\curlyvee (U_\alpha\curlyvee U_{\alpha+1})=(V_\beta\curlyvee U_\alpha)\curlyvee U_{\alpha+1}=A\curlyvee U_{\alpha+1}=B=C\succeq V_{\beta+1}.
  $$
On the other hand, by reasonings as in the proof of Lemma~\ref{lm:cross}, $\mu(\Gscr_1)\curlywedge \mu(\Gscr_2)=Y$. It follows that $U_{\alpha+1}\curlywedge V_{\beta+1}=Y\prec V_\beta$.

So we obtain $(V_\beta\curlyvee U_{\alpha+1})\curlywedge V_{\beta+1}=V_{\beta+1}$, and at the same time, $V_\beta\curlyvee (U_{\alpha+1}\curlywedge V_{\beta+1})=V_{\beta}$. This contradicts~\refeq{XiYXi} for $V_\beta/V_{\beta+1}$ and $U_{\alpha+1}$, yielding the result.
   \end{proof}
   
\noindent\textbf{Remark 2.}
Assume that the lattice $(\Lscr,\prec)$ is given via an oracle that, being asked of an element $X\in\Lscr$, outputs the set $I(X)$ of immediate successors of $X$ in the lattice. Then arguing as above, one can explicitly construct the set $\Mscr$ of maximal elements of all prime ideals and the corresponding partial order on $\Mscr$, obtaining a poset isomorphic to $(\frakG,\lessdot)$. A procedure for this task is as follows. 

Initially, using oracle calls, we construct, step by step, an arbitrary maximal chain $\Cscr_0=(X_0,X_1,\ldots,X_n)$. This gives the set of p-factors $\{X_i/X_{i-1}\colon i=1,\ldots,n\}$, and for each $i$, we find the maximal element $M_i:={\rm MAX}(\Pscr_i)$ of the prime ideal $\Pscr_i$ corresponding to $X_i/X_{i-1}$  (i.e. satisfying $X_{i-1}\in\Pscr_i\not\ni X_i$). This is reduced to constructing two sequences $X'_0,X'_1,\ldots$ and $Y'_0,Y'_1,\ldots$ as follows. Put $X'_0:=X_{i-1}$ and $Y'_0:=X_i$. At the first step, if $I(X'_0)$ consists of a unique element (namely, $Y'_0=X_i$), then $X'_0$ is just the required $M_i$. Otherwise take as $X'_1$ an arbitrary element in $I(X'_0)$ different from $Y'_0$, and put $Y'_1:=X'_1\curlyvee Y'_0$. Then $Y'_1/X'_1$ belongs to the girdle $\Gscr_i:=\Gscr(\Pscr_i)$. At the second step we handle $X'_1,Y'_1$ in a similar way, either obtaining $M_i=X'_1$ (when $|I(X'_1)|=1$), or constructing the next elements $X'_2$ and $Y'_2$. And so on, until at some, $p$-th say, step, we obtain $X'_p$ with $|I(X'_p)|=1$, which is just $M_i$.

Now the desired partial order on $\Mscr$ is constructed by appealing to Lemma~\ref{lm:immed_suc}. Namely, for each $i\in[n]$, take the single element $Y$ of $I(M_i)$, compute the set $I(Y)$, say, $\{Z_1,\ldots,Z_k\}$, and for each $j\in[k]$, acting as above, find the maximal element $M_{\alpha(j)}$ of the prime ideal $\Pscr_{\alpha(j)}$ corresponding to $Z_j/Y$ (i.e. such that  $Y\in\Pscr_{\alpha(j)}\not\ni Z_j$). Then $M_{\alpha(1)},\ldots,M_{\alpha(k)}$ are the immediately successors of $M_i$ in the poset on $\Mscr$.


\subsection{Representation for the principal lattice} \label{ssec:principal_lat}
    
Now we use the constructions and results from the previous subsection to obtain a poset representation for the principal lattice $(\Sscr^\ast,\prec^\ast=\prec_F^\ast)$. For convenience, when dealing with merely principal g-matchings, if $y\in\Sscr^\ast$ is obtained from $x\in\Sscr^\ast$ by applying a rotation $R$ with the corresponding maximal weight $\tau_R(x)$, we will briefly say that $y$ is obtained by applying $R$ to $x$ (as though ignoring the weight $\tau_R(x)$). 

Accordingly, in a full route $\Tscr=(\xmin=x_0, x_1,\ldots, x_N=\xmax)$, each $x_i$ ($i>0$) is regarded as an immediate successor of $x_{i-1}$. So $\Tscr$ is a maximal chain in the finite distributive lattice $(\Sscr^\ast,\prec^\ast)$, and using notions as above, we can work in terms of p-factors (such as $x_i/x_{i-1}$), prime ideals $\Pscr_i$ and their associated girdles $\Gscr(\Pscr_i)$. 

First of all we refine properties of girdles. Let $\Hscr=(\Sscr^\ast,\Escr)$ be the directed graph where the edge set $\Escr$ is formed by the pairs $(x,y)$ such that $y$ is an immediate successor of $x$ (equivalently, $y/x$ is a p-factor), i.e. $\Hscr$ is the Hasse diagram for $(\Sscr^\ast,\prec^\ast)$. Then each girdle $\Gscr=\Gscr(\Pscr)$ corresponds to the \emph{directed cut} formed by the edges $(x,y)$ going from $\Pscr$ to the filter $\Sscr^\ast-\Pscr$. Moreover, by~\refeq{distinct_cleav}, these edges have no common endvertices. Let us associate with each p-factor $y/x$ (and the edge $(x,y)$ of $\Hscr$) the rotation $R\in\Rscr$ such that $y$ is obtained by applying $R$ to $x$; we denote $R$ as $\rho(y/x)$. Then the following property of homogeneity takes place.
 \begin{lemma} \label{lm:homogen}
For all p-factors $y/x$ in a girdle $\Gscr(\Pscr)$, the rotations $\rho(y/x)$ are the same. In other words, all edges of the cut $(\Pscr,\Sscr^\ast-\Pscr)$ in $\Hscr$ correspond to the same rotation.
  \end{lemma}
  \begin{proof}
(Cf.~Remark~2.) For $y/x\in\Gscr(\Pscr)$, consider a directed path $(x_0,x_1,\ldots,x_k)$ in $\Hscr$ going from $x=x_0$ to the maximal element ${\rm MAX}(\Pscr)=x_k$. Let $R$ and $R'$ be the rotations corresponding to the edges $(x,y)$ and $(x,x_1)$, respectively. These $R$ and $R'$ commute at $x$ (cf. Corollary~\ref{cor:commute}), whence $\Hscr$ has vertex $y_1\in\Sscr^\ast$ and edges $(x_1,y_1)$ and $(y,y_1)$, the former (latter) corresponding to the rotation $R$ (resp. $R')$. (Here $y_1$ is obtained from $x$ by applying $R$, followed by applying $R'$, or conversely). 

Since $y$ belongs to the filter $\Dscr:=\Sscr^\ast-\Pscr$, so does $y_1$. Then $y_1/x_1\in\Gscr(\Pscr)$ and $\rho(y_1/x_1)=\rho(y/x)=R$. Arguing in a similar way for $y_1/x_1$, and so on, we eventually come to $y_k$ such that $y_k/x_k\in\Gscr(\Pscr)$ and $\rho(y_k/x_k)=R$, whence the result follows.  
  \end{proof}

Using this lemma, we associate to each girdle $\Gscr$ the rotation $\rho(y/x)$ for $y/x\in \Gscr$; denote it as $R(\Gscr)$. It is possible that the same rotation is associated to several girdles. More precisely, by reasonings in Sects.~\SEC{addit_prop} and~\SSEC{birk}, for each rotation $R\in\Rscr$, there are as many girdles $\Gscr$ with $R(\Gscr)=R$ as the number of occurrences of this $R$ in the sequence $\Rscr(\Tscr)$ of rotations used in a full route $\Tscr$. Denote this number by $\eta(R)$ (equal to $\omega^{\xmax}(R)$, cf. Lemma~\ref{lm:omega_x}) and denote the set of these girdles by $\frakG_R$. Considering a full route $\Tscr$, we can label the members of $\frakG_R$ as $\Gscr^{(1)},\ldots, \Gscr^{(\eta(R))}$ according to their order in the sequence $\Rscr(\Tscr)$, i.e. $\Gscr^{(i)}$ corresponds to $i$-th occurrence of $R$ in $\Rscr(\Tscr)$. 

The next useful observation immediately follows from Lemma~\ref{lm:cross}:
 \begin{numitem1} \label{eq:order_girdles}
for each rotation $R\in\Rscr$, the labeling $\Gscr^{(1)},\ldots,\Gscr^{(\eta(R))}$ of members of $\frakG_R$ is the same for all full routes; equivalently, $\Gscr^{(1)}\lessdot\cdots\lessdot \Gscr^{(\eta(R))}$.
 \end{numitem1}
Indeed, if $\Gscr,\Gscr'\in\frakG_R$ are incomparable, then, by Lemma~\ref{lm:cross}, there are $x,y,z\in\Sscr^\ast$ forming p-factors $y/x\in\Gscr$ and $z/x\in\Gscr'$. This means that $x,y$ are consecutive elements in some full route and $x,z$ are consecutive elements in another full route. But then both $y,z$ are obtained from $x$ by using the same $R$; a contradiction.
 \smallskip
 
By~\refeq{order_girdles}, for each rotation $R\in\Rscr$, we can label the copies of $R$ as $R^{(1)},\ldots,R^{(\eta(R))}$ (ordered according to their appearance in the sequence $\Rscr(\Tscr)$ for a full route $\Tscr$), i.e. $R^{(i)}$ corresponds to $i$-th girdle $\Gscr^{(i)}$ in $\frakG_R$. We denote the set $\{R^{(1)},\ldots,R^{(\eta(R))}\}$ as $\Uscr_R$, and denote the union of these sets over all different rotations $R\in\Rscr$ as $\Uscr$.
We now can transfer, in a natural way, the partial order $\lessdot$ on the girdles to the set of labeled rotations $\Uscr$, keeping the same symbol $\lessdot$. Namely,
  \begin{numitem1} \label{eq:order_rotat}
if girdles $\Gscr_\alpha^{(i)},\Gscr_\beta^{(j)}\in\frakG$ satisfy $\Gscr_\alpha^{(i)}\lessdot \Gscr_\beta^{(j)}$, then we compare their corresponding labeled rotations $R_\alpha^{(i)},R_\beta^{(j)}$ as $R_\alpha^{(i)}\lessdot R_\beta^{(j)}$.
  \end{numitem1}
  
This correspondence gives an isomorphism between $(\frakG,\lessdot)$ and $(\Uscr,\lessdot)$, and now Proposition~\ref{pr:poset_repr} implies the following result.
  \begin{theorem} \label{tm:isomorph_princ}
For $x\in\Sscr^\ast$, let $\phi^\ast(x)$ be the set of labeled rotations  used in a principal route from $\xmin$ to $x$. Then $\phi^\ast$ establishes an isomorphism between the principal lattice $(\Sscr^\ast,\prec^\ast)$ and the lattice of closed sets in the poset of labeled rotations $(\Uscr,\lessdot)$. In particular, $\phi^\ast(\xmin)=\{\emptyset\}$ and $\phi^\ast(\xmax)=\Uscr$.
  \end{theorem}
  
It remains to explain how to construct the poset $(\Uscr,\lessdot)$ explicitly. We assume availability of the initial stable g-matching $\xmin$ (details of finding it are left to Sect.~\SSEC{xmin}). We proceed by a method as in Remark~2. More precisely, using techniques from Sect.~\SEC{act-rot}, we first construct a full route $\Tscr_0=(\xmin=x_0,x_1,\ldots,x_N=\xmax)$. Along the way, we obtain the set $\Rscr$ of different rotations and the set $\Uscr$ of labeled rotations, where in each collection $\Uscr_R$, the labels are assigned chronologically (which does not depend on the route). Next we use $N$ iterations to compute the immediately succeeding relations on $\Uscr$ (which is already done within each collection $\Uscr_R$), thus obtaining the Hasse diagram for $(\Uscr,\lessdot)$, as required.

At $i$-th iteration, we treat the g-matching $x_{i-1}$ along with the labeled rotation, say, $R_\alpha^{(j)}$, that corresponds to the p-factor $x_i/x_{i-1}$ (i.e. $x_i=x_{i-1}+\tau_{R_\alpha}(x_{i-1})\chi^{R_\alpha}$). The iteration consists of two stages. 
 \smallskip
 
At the \emph{1st stage}, acting as in Remark~2, we find the element $\mu(\Gscr)$ for the girdle $\Gscr=\Gscr_\alpha^{(j)}$ containing $x_i/x_{i-1}$ (viz. the maximal element of the prime ideal associated with $\Gscr$). Then $x:=\mu(\Gscr)$ has exactly one applicable rotation $R$, and for $y$ obtained from $x$ by applying $R$ (with weight $\tau_R(x)$), the p-factor $y/x$ belongs to the same girdle $\Gscr$. So $R$ should be identified with the labeled rotation $R_\alpha^{(j)}$.
\smallskip

At the \emph{2nd stage}, we find the set of labeled rotations that are the immediate successors to $R_\alpha^{(j)}$ (by appealing to Lemma~\ref{lm:immed_suc}). To this aim, we take the g-matching $y$ as above and  find the set $\Rscr(y)$ of rotations applicable to $y$. Then for each $R_\beta\in\Rscr(y)$, we compute the number of times it is used in a principal route $\Tscr_i$ from $\xmin$ to $x=\mu(\Gscr)$ (such a $\Tscr_i$ is the concatenation of the part of $\Tscr_0$ from $\xmin$ to $x_{i-1}$ and a principal route from $x_{i-1}$ to $x$ constructed at the 1st stage). This just gives the appropriate label for $R_\beta$, say, $r$, and we obtain the labeled rotation $R_\beta^{(r)}$ which is an immediate successor of  $R_\alpha^{(j)}$ in $(\Uscr,\lessdot)$. (This $R_\beta^{(r)}$ corresponds to the girdle containing the p-factor $z/y$, where $z$ is obtained from $y$ by applying $R_\beta$ with weight $\tau_{R_\beta}(y)$.) 
 \smallskip

Upon termination of $N$ iterations, we obtain the desired poset $(\Uscr,\lessdot)$ and can conclude with the following.
 \begin{corollary} \label{cor:princ_poset}
In the above procedure, the task of finding the rotational poset $(\Uscr,\lessdot)$ to represent the lattice $(\Sscr^\ast,\prec^\ast)$ is reduced to constructing $N+1$ principal routes, where $N$ is the length (number of links) of a full route.
  \end{corollary}


\subsection{Representation for the complete lattice} \label{ssec:complete_lat}

Now we transform the representation for the principal lattice exposed in Theorem~\ref{tm:isomorph_princ} into a one for the lattice $(\Sscr,\prec)$ of all stable g-matchings. By reasonings in Sect.~\SEC{addit_prop}, for each stable g-matching $x\in\Sscr$, there is a non-excessive route $\Tscr=(x_0,x_1,\ldots,x_N)$ from $x_0=\xmin$ to $x_N=x$ (see~\refeq{non-excess}), where each $x_i$ is obtained from $x_{i-1}$ by applying a rotation $R_i\in\Rscr(x_{i-1})$ with a weight $\lambda_i\le \tau_{R_i}(x_{i-1})$, and the family of pairs $(R_i,\lambda_i)$, with possible replications, does not depend on the route (cf. Proposition~\ref{pr:invar_rot}). 

Recall that $\Tscr$ is meant to be non-excessive if for any two equal rotations $R_j,R_i$ (coinciding as alternating cycles in $G$), the relations $j<i$ and $\lambda_i>0$ imply $\lambda_j=\tau_{R_j}(x_{j-1})$. From Corollary~\ref{cor:commute} it follows that if  $\lambda_i$ with $0<\lambda_i<\tau_{R_i}(x_{i-1})$ is changed to any integer value within the interval $[0,\tau_{R_i}(x_{i-1})]$, keeping the weights of other rotations, then we again obtain a correct non-excessive route. Due to this, one can associate to $x$ the principal g-matching $x^\ast$ generated by the principal route $\Tscr^\ast$ obtained from $\Tscr$ by replacing each weight $\lambda_i>0$ by $\tau_{R_i}(x_{i-1})$.

Under such updates, each non-excessive route from $\xmin$ to $x$ one-to-one corresponds to a principal route from $\xmin$ to $x^\ast$, and the machinery of girdles and labeled rotations developed above can be straightforwardly transferred to work with any stable g-matching $x\in\Sscr$. On this way, in the set $\Rscr$ of different rotations, for each $R_\alpha\in\Rscr$, we obtain all copies $R_\alpha^{(i)}$ of $R_\alpha$ (related to different girdles and forming the collection $\Uscr_{R_\alpha}$), along with the corresponding maximal weights that we denote as $\tau(R_\alpha^{(i)})$. (These weights are computed in the process of constructing a full route $\Tscr_0$.) As a result, the previous poset $(\Uscr,\lessdot)$ for $(\Sscr^\ast,\prec^\ast)$ turns into its extension with weighted labeled rotations, denoted as $(\Uscr,\tau,\lessdot)$. Now, instead of closed subsets of $\Uscr$, we should consider special functions on $\Uscr$.
 \medskip

\noindent\textbf{Definition 9.} A function $\zeta:\Uscr\to\Zset_+$ is called \emph{closed} if for any labeled rotations $R_\alpha^{(i)},R_\beta^{(j)}\in \Uscr$ such that $R_\alpha^{(i)}\lessdot R_\beta^{(j)}$ and $\zeta(R_\beta^{(j)})>0$, there holds $\zeta(R_\alpha^{(i)})=\tau(R_\alpha^{(i)})$.
 \medskip

Such functions form a lattice under the standard comparison $<$. For $x\in\Sscr$, consider a non-excessive route $\Tscr$ from $\xmin$ to $x$ and define
 \begin{numitem1} \label{eq:phi_x}
$\phi(x)$ to be the function $\zeta:\Uscr\to\Zset_+$ that takes value $\zeta(R_\alpha^{(i)})=\lambda$ if $R_\alpha^{(i)}$ is a labeled rotation used with weight $\lambda$ in $\Tscr$, and value 0 if $R_\alpha^{(i)}$ is not used in $\Tscr$.
  \end{numitem1}
  
In light of explanations above, $\phi(x)$ is closed for all $x\in\Sscr$, and now comparing $\phi$ with the map $\phi^\ast$ as in Theorem~\ref{tm:isomorph_princ}, we can conclude with the following
  \begin{theorem} \label{tm:isomorph_complete}
The map $\phi$ as in~\refeq{phi_x} establishes an isomorphism between the lattice of stable g-matchings $(\Sscr,\prec)$ and the lattice of closed functions for the poset $(\Uscr,\tau,\lessdot)$.
  \end{theorem} 
  
In fact, the method of finding $(\Uscr,\lessdot)$ described in Sect.~\SSEC{principal_lat} is suitable (up to minor details) to construct the weighted rotational poset $(\Uscr,\tau,\lessdot)$ representing $(\Sscr,\prec)$.


\section{Computational aspects} \label{sec:comput}

In this section we specify details of the method of constructing the weighted rotational poset $(\Uscr,\tau,\lessdot)$ described in the previous section, and estimate its complexity. 

As is said in the Introduction, we assume that for each vertex $v\in V$, the CF $C_v$ is given via an \emph{oracle} that, being asked of a vector $z\in\Bscr_v$, outputs its ``acceptable part'' $C_v(z)$. We conditionally assume that each \emph{oracle call} for $C_v$  takes time polynomial in $|V|$ (possibly even $O(|E_v|)$), and in fact we are mostly interested in estimating the number of oracle calls rather than the  complexity of their implementations. 

Basic tasks used in our constructions are:
  \begin{numitem1} \label{eq:find_R(x)}
given a stable g-matching $x\in\Sscr$, find the set $\Rscr(x)$ of rotations applicable to $x$;
  \end{numitem1} 
  \begin{numitem1}\label{eq:find_tau}
given $x\in\Sscr$ and $R\in\Rscr(x)$, find the maximal feasible weight $\tau_R(x)$.
  \end{numitem1}
  
To solve~\refeq{find_R(x)}, we construct the active graph $\Gamma(x)$ (see Definition~5 in Sect.~\SSEC{act_graph}) and decompose it into corresponding edge-simple cycles (viz. rotations), obtaining $\Rscr(x)$. We encode a rotation as the corresponding sequence $(a_1,c_1,a_2,\ldots,a_k,c_k)$ of edges such that for $i\in[k]$, $(a_i,c_i)$ is a legal $F$-pair, and $(c_i,a_{i+1})$ is an essential $W$-pair, letting $a_{k+1}:=a_1$. (Then any two rotations can be easily compared, when needed, in $O(|E|)$ time.) To construct $\Gamma(x)$, we first find all legal $F$-pairs, which is reduced to scanning all edges $e=wf\in E$, computing the vectors $C_f(x_f+\onebf^e)$ and comparing them with $x_f+\onebf^e$. On this way, we also obtain the sets $U^+_F$ and $U^-_F$ (see Definition~2 in Sect.~\SSEC{act_graph}). This takes $O(|E|)$ oracle calls. Then, scanning the elements $U^-_F$, we find all essential $W$-pairs. This can be organized so as to take $O(|E|^2)$ oracle calls in total. (Here, for $c=wf\in U_F^-$ fixed, we first extract all edges $a\in E_w\cap U_F^+$ satisfying  $C_w(x_w-\onebf^c+\onebf^a)=x_w-\onebf^c+\onebf^a$. Then in the set $A$ of extracted elements, we find the (unique) element $a$ such that $C_w(x_w-\onebf^c+\onebf^a+\onebf^d)=x_w-\onebf^c+\onebf^a$ for all other elements $d\in A$, obtaining the desired essential pair $(c,a)$. This can be done by $O(|E_w)|)$ oracle calls; cf. Lemma~\ref{lm:essent_pair}.)
Thus, to solve~\refeq{find_R(x)} takes $O(|E|^2)$ oracle calls. 
  \smallskip

To solve~\refeq{find_tau}, we act straightforwardly by constructing, step by step, a sequence $x=x_0,x_1,\ldots,x_k$ of stable g-matchings by use of the same  rotation $R$. Here for each $i>0$, we examine whether all legal $F$-pairs and all essential $W$-pairs occurring in $R$ preserve for $x_i:=x+i\chi^R$, and if so, form the next $x_{i+1}$ by applying $R$ with weight 1 to $x_i$.  The weight $\tau_R(x)$ is equal to the maximal $i$ for which either we reach the upper or lower bound on some edge in $R$, or at least one legal or essential pair becomes destroyed. Since $|R^+|=|R^-|=O(|E|)$, and $\tau_R(x)\le\bmax$, we can estimate the number of oracle calls (roughly) as $O(|E|^2\,\bmax)$. Here, as before, $\bmax:=\max\{b(e)\colon e\in E\}$.
 \medskip
 
\noindent\textbf{Remark 3.}
As is shown in~\cite[Sect.~4]{karz3}, for the stable generalized allocation problem, task~\refeq{find_tau} can be solved in weakly polynomial time, using $O(|E|\,\log \bmax)$ oracle calls. However, we do not see how to reduce the factor $\bmax$ to $\log\bmax$ in the complexity of solving~\refeq{find_tau} for a general case of SGMM.


\subsection{Constructing the poset $(\Uscr,\tau,\lessdot)$} \label{ssec:constr_poset}

To construct this poset, we follow the approach outlined in Sect.~\SSEC{principal_lat}, assuming that the initial stable g-matching $\xmin$ is available. According to it, our algorithm consists of $N+1$ \emph{stages}, where $N$ is the length of a full route.

At \emph{Stage 0}, we construct an arbitrary full route $\Tscr_0=(\xmin=x_0,x_1,\ldots,x_N=\xmax)$. During this stage, we maintain a current list $\Lscr$ of different rotations used in the already constructed part of the route, and for each rotation $R\in\Lscr$, we maintain the number $\eta(R)$ of occurrences of $R$ in this part. For $i=1,\ldots,N$, let $R[i]$ denote the rotation applied to $x_{i-1}$ (with weight $\tau_{R[i]}(x_{i-1})$) to obtain $x_i$. Initially, we set $\Lscr:=\emptyset$. 

At $i$-th step of this stage, we handle $x_{i-1}$, and find rotation $R[i]$ and weight~$\tau[i]:=\tau_{R[i]}(x_{i-1})$ by solving~\refeq{find_R(x)} and~\refeq{find_tau}. (Note that it is not necessary to update the current active graph $\Gamma$ at each step. For if at some step, we know $\Gamma$ and its decomposition into rotations, say, into $p$ rotations, then we can use these rotations in the next $p$ steps of the stage, updating $\Gamma$ after that.) 

Also at $i$-th step, we check if $R[i]$ is already present in the list $\Lscr$. Acting straightforwardly, we can simply compare $R[i]$ with each member of $\Lscr$; this takes $O(|E| |\Lscr|)$, or $O(|E| N)$), usual operations. (A faster procedure, based on dynamic (or balanced) tree techniques, can be arranged so as to take $O(|E| \log \Lscr)$ amortized time.)
If $R[i]\notin \Lscr$, then we add $R[i]$ to $\Lscr$ and assign $\eta(R[i]):=1$. And if $R[i]$ was added earlier to $\Lscr$, then the number $\eta$ for $R[i]$ in $\Lscr$ is increased by one. This rotation is labeled by this $\eta$ and added to the current sequence of weighted labeled rotations. 

Acting so at each step of Stage~0, we eventually obtain a full route $\Tscr_0$, the  set $\Rscr$ of all different rotations, and the corresponding sequence $\Pi$ of weighted labeled rotations $(R^{(j)},\tau_{R^{(j)}})$. Summarizing explanations concerning complexities, we obtain that
  \begin{numitem1} \label{eq:constr_T0}
to construct $\Tscr_0$, the number of oracle calls is estimated as $O(\bmax|E|^2 N )$, and the number of other operations as $O(\bmax N+N\log N )$ times a polynomial in $|V|$. 
 \end{numitem1}
 
Thus, Stage~0 outputs the set $\Uscr$ of labeled rotations. The aim of other $N$ stages is to establish the partial order $\lessdot$ on $\Uscr$. Here $i$-th stage ($1\le i\le N$) finds the set of immediate successors of $i$-th labeled rotation in the sequence $\Pi$. For definiteness, we denote this labeled rotation as $R[i]^{\ell(i)}$. 

\emph{Stage~$i$} starts with the part of the route $\Tscr_0$ from $\xmin$ to $x_{i-1}$, and continues by making ``principal moves'', aiming to reach the maximal element ${\rm MAX}[i]$ of the prime ideal of $(\Sscr^\ast,\prec^\ast)$ corresponding to the girdle formed by the copies of $R[i]^{\ell(i)}$. Here we follow the description in Sect.~\SSEC{principal_lat}; namely, at each step, for a current g-matching $x$, in the active graph $\Gamma(x)$ we choose a rotation  $R$ different from $R[i]$. We finish when such an $R$ does not exist; then $x$ is just the required ${\rm MAX}[i]$. Next we form $y\in\Sscr^\ast$ by applying to $x$ the rotation $R[i]$ with weight $\tau_{R[i]}(x)$, and then construct the set of rotations applicable to $y$, say, $\Rscr(y)=\{L_1,\ldots,L_q\}$. Let $\Tscr_i$ be the constructed principal route from $\xmin$ to $y$ (which extends the part of $\Tscr_0$ from $\xmin$ to $x_{i-1}$). For each $L_p\in\Rscr(y)$, moving along $\Tscr_i$, we count the number $\eta(L_p)$ of occurrences of this rotation in it. Then $L_p$ labeled $\eta(L_p)+1$ is an immediate successor (by $\lessdot$) of $R[i]^{\ell(i)}$, and all its immediate successors are obtained in this way (cf. 2nd stage in Sect.~\SSEC{principal_lat}).

Upon termination of the last, $N$-th, stage, we obtain the desired Hasse diagram of $(\Uscr,\lessdot)$ along with the maximal weights $\tau$ of all labeled rotations in $\Uscr$. The numbers of oracle calls and other operations can be estimated rather straightforwardly (cf.~\refeq{constr_T0}), and we can conclude with the following
  \begin{theorem} \label{tm:gen_time}
Provided that $\xmin$ is known, the above algorithm constructs the weighted rotational poset $(\Uscr,\tau,\lessdot)$ for which the lattice of closed function is isomorphic to the lattice $(\Lscr,\prec)$ of stable g-matchings. The size $|\Uscr|$ is equal to the length $N$ of a full route (viz. a maximal chain in $(\Sscr^\ast,\prec^\ast)$). The number of oracle calls in the algorithm can be estimated as $O(|E|^2(\bmax N+N^2))$, and the number of other operations as $O(\bmax N+N^2\log N)$ times a polynomial in $|V|$.
\end{theorem}

Now we can estimate $N$, yielding pseudo polynomial time bounds in Theorem~\ref{tm:gen_time}.

 \begin{lemma} \label{lm:N}
$N\le \bmax |E|/2$.
  \end{lemma}
  \begin{proof}
Let $\xmin=x^0,x^1,\ldots,x^{N'}=\xmax$ be a route where each $x^i$ is obtained by shifting $x^{i-1}$ along a rotation with weight 1. Then $N'\ge N$, and for each edge $e$, we have $|x^i(e)-x^{i-1}(e)|\le 1$. We assert that the sequence $x^0(e),x^1(e), \ldots,x^{N'}(e)$ behaves in ``one peak'' manner: there is some $0\le i\le N'$ such that the values from $x^0(e)$ to $x^i(e)$ are monotone non-decreasing, while from $x^i(e)$ to $x^{N'}(e)$ are monotone non-increasing. Then there are at most $2b(e)$ changes for $e$ during the process, and the result follows since one application of a rotation changes the values on at least four edges.

Suppose that the assertion is not valid for some $e=wf$. Then there are $p<q<r$ such that $x^q(e)+1=x^p(e)=x^r(e)=:\alpha$. One may assume that $q=p+1$; then $x^q$ is obtained from $x^p$ by applying a rotation $R$ such that $e\in R^-$. We have $x^p_f\prec_f x^q_f\prec_f x^r_f$. Let $z:=x^q_f\vee x^r_f$ and $z':=x^q_f+\onebf_f^e$. Then $z\ge z'$ (since $x^q(e)+1=x^r(e)$) and $C_f(z)=C_f(x^r_f\vee x^q_f)=x^r_f$. Applying~(A2) to $z\ge z'$, we have  $C_f(z)\wedge z'\le C_f(z')$. But $C_f(z)(e)=x^r_f(e)=\alpha$ and $z'(e)=\alpha$, whereas $C_f(z')(e)=\alpha-1$ (since $e\in R^-$, whence $e$ is not interesting for $f$ under $x^q_f$, by Lemmas~\ref{lm:acd}(iii) and~\ref{lm:acac}). A contradiction.  
  \end{proof}


\subsection{Finding the initial stable g-matching $\xmin$.} \label{ssec:xmin}

We rely on a method in~\cite[Sec.~3.1]{AG} intended for a general stability model on bipartite graphs considered there. Below we outline that method regarding our model SGMM. This gives a \emph{pseudo polynomial} algorithm for finding $\xmin$. 

The algorithm consists of iterations which construct triples $(b^i,x^i,y^i)$ of functions on $E$. Initially, one puts $b^0:=b$. In the input of  $i$-th iteration, there is a function $b^i\in \Zset_+^E$ (already known). The iteration consists of two stages.
 \smallskip

At the \emph{1st stage} of $i$-th iteration, $b^i$ is transformed into
$x^i\in\Zset_+^E$ as follows. For each $w\in W$, we apply the CF $C_w$ to the restriction $b^i_w$ of $b^i$ to $E_w$, i.e. assign $x^i_w:=C_w(b^i_w)$. Then $x^i$ is acceptable for all vertices in $W$ (but not necessarily in $F$).
 \smallskip

At the \emph{2nd stage} of $i$-th iteration, $x^i$ is transformed into $y^i\in
\Zset_+^E$ by applying, for each $f\in F$, the CF $C_f$ to the restriction $x^i_f$ of $x^i$ to $E_f$, i.e. by setting $y^i_f:=C_f(x^i_f)$. Then
$y^i$ is acceptable for all vertices; however, $y^i$ need not be stable.

The functions $b^i,x^i,y^i$ are then used to form the input function
$b^{i+1}$ for the next iteration. This is done by the following rule:
  \begin{numitem1} \label{eq:Bi+1}
~for each $e\in E$, (a) if $y^i(e)=x^i(e)$, then $b^{i+1}(e):=b^i(e)$; and (b) if  
$y^i(e)<x^i(e)$, then $b^{i+1}(e):=y^i(e)$.
  \end{numitem1}

Then $b^{i+1}$ is transformed into $x^{i+1}$ and further into $y^{i+1}$ as described above. And so on. The process terminates when at a current, $p$-th say, iteration, we obtain the equality $y^p=x^p$ (equivalently, $b^{p+1}=b^p$). By~\refeq{Bi+1}, we have $b^0> b^1>\cdots>b^p$. Therefore, the process is finite and the number of iterations does not exceed $|E|\,\bmax$. Results in~\cite{AG} imply the following
 \begin{prop} \label{pr:Xp}
The final g-matching $x^p$ constructed by the above algorithm is stable and optimal for $W$, i.e. $x^p=\xmin$.
 \end{prop}

(It should be noted that for a general model in~\cite{AG}, where the domain $\Bscr$ can be infinite, the process of iteratively transforming the current functions may last infinitely long; however, it always converges to some triple $(\hat b,\hat x,\hat y)$ satisfying $\hat x=\hat y$. It is proved in~\cite[Ths.~1,2]{AG} that the limit function $\hat x$ is stable and optimal for $W$.)

In our particular case, with integer-valued functions, the process is always finite, yielding Proposition~\ref{pr:Xp}. The algorithm takes $O(|E|\,\bmax)$ oracle calls, and the number of usual operations is bounded by $\bmax$ times a polynomial in $|V|$. So the whole algorithm of finding the poset $(\Uscr,\tau,\lessdot)$ has complexity as is pointed out in Theorem~\ref{tm:gen_time}.


\section{A polynomial case} \label{sec:special}

In this section we impose an additional condition on all choice functions $C_v$, $v\in V$, called the \emph{gapless condition}. This reads as follows:
  \begin{itemize}
\item[(C):]
for $v\in V$, if acceptable vectors $z^1,z^2,z^3\in\Ascr_v$ and edges $a,c^1,c^2,c^3\in E_v$ satisfy the relations: (i) $z^1\prec_v z^2\prec_v z^3$, (ii) $C_v(z^i+\onebf_v^a)=z^i+\onebf_v^a-\onebf_v^{c^i}$ for $i=1,2,3$, and (iii)~$c^1=c^3$, then $c^1=c^2$ as well.
  \end{itemize}

A similar condition was introduced in~\cite{karz3} in frameworks of the stable generalized allocation model. Below, extending results there to our model SGMM, we show that condition~(C) guarantees that all rotations in a full route are different and the number of these is polynomial in $|V|$. As is mentioned in the Introduction, a simple example of validity of~(C) concerns $C_v$ generated by a linear order $>_v$ on $E_v$, i.e. when we deal with the standard allocation model of Baiou and Balinsky. (In this case, relations~(i),(ii) imply $c^1\le_v c^2\le_v c^3$, yielding the desired condition.) Also~(C) trivially holds when $b(e)\le 2$ for all $e\in E$.
We prove the following
 \begin{theorem} \label{tm:condC}
Suppose that (C) is valid and let $\Tscr=(x^1,\ldots,x^N)$ be a non-excessive
route. Then: {\rm(i)} any rotation occurs in $\Tscr$ at most once; and
{\rm(ii)} $N<4|E|^2|F| |W|$.
  \end{theorem}
  \begin{proof}
~We first show~(i). Let each $x^{i+1}$ be obtained from $x^i$ by applying a rotation $R_i$ with weight $\lambda_i$. Suppose that $R_i=R_k=:R$ for some $i<k<N$. The case $k=i+1$ is impossible (otherwise one can apply to $x^i$ the rotation $R$ with weight $\lambda_i+\lambda_{i+1}$, contrary to the non-excessiveness of $\Tscr$). So one may assume that $k\ge i+2$.

Consider a legal $f$-pair $(a,c)$ in $R$, where $f\in F$. Then $C(x_f^i+\onebf^a)=x_f^i+\onebf^a-\onebf^c$ and $C(x_f^k+\onebf^a)=x_f^k+\onebf^a-\onebf^c$, where $C:=C_f$, and $\onebf^\bullet:=\onebf_f^\bullet$. Since $x^1\prec_F\cdots\prec_F x^N$, we have $x_f^i\prec_f x_f^{i+1}\preceq_f \cdots \preceq_f x_f^k\prec_f x_f^{k+1}$. We assert that
  \begin{numitem1} \label{eq:xj_a_interest}
for each $j=i+1,\ldots,k-1$, (a) the edge $a$ as above is interesting for $f$ under $x^j$, and (b)~$C(x_f^j+\onebf^a)=x_f^j+\onebf^a-\onebf^c$.
  \end{numitem1}

To see (a), suppose that $x^j(a)\le x^k(a)$ for some $i<j<k$. Then $x^j(a)< x^{k+1}(a)$, and the fact that $a$ is interesting for $f$ under $x^j$ follows from Lemma~\ref{lm:e_interest}(i),(ii) (applied to $x_f^j$, $x_f^{k+1}$ and $a$). Now suppose that there is $i<j<k$ such that $x^j(a)>x^k(a)$, and take the maximal $j$ under this property. Then $x^j(a)> x^{j+1}(a)$. This is possible only if the rotation $R_{j}$ has a legal $f$-pair of the form $(a',a)$ (in which case $x^{j+1}(a)=x^j(a)-\lambda_j$). Then, by  Lemma~\ref{lm:acd}(iii) (with $(a',a)$ in place of $(a,c)$), ~$a$ is not interesting for $f$ under $x_f^j+\lambda_j(\onebf^{a'}-\onebf^a)$. This implies (using Lemma~\ref{lm:acac}) that $a$ is not interesting for $f$ under $x_f^{j+1}$ as well, obtaining a contradiction since $x^{j+1}(a)\le x^k(a)$. Thus, assertion (a) in~\refeq{xj_a_interest} is valid. 

To see~(b), consider $C(x_f^j+\onebf^a)$ for $i<j<k$. It is equal to either (1) $x_f^j+\onebf^a-\onebf^{c'}$ for some $c'\in E_f-\{a\}$, or (2) $x_f^j+\onebf^a$. In case~(1), condition~(C) applied to $x_f^i\prec_f x_f^j\prec_f x_f^k$ gives $c'=c$, as required. And in case~(2), for $y:=x_f^j\vee x_f^k$ and $p:=x^k(a)$, we have
  $$
C(y+\onebf^a)=C(x_f^j\vee (x_f^k+\onebf^a))=C(C(x_f^j\vee x_f^k)\vee (p+1)\onebf^a)=C(x_f^k+\onebf^a)=x_f^k+\onebf^a-\onebf^c.
$$
Then $|C(y+\onebf^a)|=|x_f^k|=|x_f^j|$. On the other hand, $C(x_f^j+\onebf^a)=x_f^j+\onebf^a$ implies $|x_f^j|<|C(x_f^j+\onebf^a)|$. This contradicts~(A3), in view of $y+\onebf^a\le x_f^j+\onebf^a$. 

Thus, \refeq{xj_a_interest} is valid, and therefore, $(a,c)$ is a legal $F$-pair for all $x^i,\ldots,x^k$.
 
To finish the proof of~(i), consider the reverse sequence $(x^N,x^{N-1},\ldots,x^1)$. It can be regarded as a route relative to the order $\prec_W$, i.e. it obeys $x^N\prec_W\cdots \prec_W x^1$, and each $x^i$ is an immediate successor of $x^{i+1}$ under $\prec_W$. Moreover, $x^i$ is obtained from $x^{i+1}$ by applying the rotation $R_i^{-1}$ reverse to $R_i$ taken with weight $\lambda_i$. Then an essential  $W$-pair $(c,a)$ for $R_i$ under $x^i$ becomes a legal pair for $R_i^{-1}$ under $x^{i+1}$. 

Now take $i<k$ such that $R_i=R_k=R$ and consider a legal $w$-pair (tandem) $(c,a)$ for $R^{-1}$, where $w\in W$. The assertion ``symmetric'' to~\refeq{xj_a_interest} says that: for each $j=k,k-1,\ldots,i+2$, (a) the edge $c$ is interesting for $w$ under $x^j$, and (b) $C_w(x_w^j+\onebf^c)=x_w^j+\onebf^c-\onebf^a$. Taking together these two assertions, we can conclude that $R_i=R_{i+1}=\cdots=R_k$, contradicting the non-excessiveness of $\Tscr$ (since $i\ne k$).

Thus, assertion~(i) in the theorem is proven.

Next,  to obtain~(ii), we argue as follows. Let $\Lscr(F)$ be the set of legal $F$-pairs $(a,c)$ occurring in $x^1,\ldots,x^N$. The total number of these pairs does not exceed $|E| (|W|-1)$ (since for each $a\in E$, there is at most $|W|-1$ pairs $(a,c)$ sharing an endvertex in $F$). Define $\Sigma(F)$ to be the collection of maximal subsets $I$ in $[N]$ such that the rotations $R_i$ with $i\in I$ have the same set of legal $F$-pairs (tandems) $(a,c)$. From assertion~(i) it follows that for each pair $(a,c)\in\Lscr(F)$, the set $K_{a,c}$ of indexes $i\in N$ such that the rotation $R_i$ uses $(a,c)$ forms  an (integer) interval in $[N]$. Then each $I\in \Sigma(F)$ is a maximal subset (interval) in $[N]$ separated by none of the intervals $K_{a,c}$ for $(a,c)\in \Lscr(F)$. This implies $|\Sigma(F)|\le 2|\Lscr(F)|-1$. 

For similar reasons, the set $\Lscr(W)$ of legal $W$-pairs $(c,a)$ occurring in $x^1,\ldots,x^N$ has size at most $|E|(|F|-1)$, and the collection $\Sigma(W)$ of maximal subsets (intervals) $J\subset [N]$ such that the rotations $R_i$ with $i\in J$ have the same set of legal $W$-pairs has size at most $2|\Lscr(W)|-1$. Now since each rotation is determined by its $F$- and $W$-legal pairs, property~(i) implies that $|I\cap J| \le 1$ for any $I\in\Sigma(F)$ and $J\in\Sigma(W)$. 

This implies $N<|E|^2|F| |W|$, as required.
  \end{proof}

Subject to the gapless condition~(C), for a rotation $R\in\Rscr(x)$ applicable to a stable g-matching $x\in\Sscr$, the maximal weight $\tau_R(x)$ can be computed in weakly polynomial time, by use of the following procedure in a ``divide-and-conquer'' manner:
 \begin{itemize}
\item[(P):]
Initially assign lower and upper bounds for $\tau_R(x)$ to be $\lambda:=1$ and
$\nu:=\min\{\min\{x(e)\colon e\in R^-\}, \min\{b(e)-x(e)\colon e\in R^+\}\}$,
respectively. Take the (sub)median integer weight $\mu$ in the interval
$[\lambda,\nu]$, namely, $\mu:=\lfloor(\lambda+\nu)/2\rfloor$, and form the vector $x':=x+\mu\chi^R$. Check whether $x'$ is stable (by using $|V|+2|E|$ oracle calls), and if so, check whether $R$ is a rotation applicable to $x'$ (by using $O(|R^+|+|R^-| |V|)$, or $O(|E| |V|)$, oracle calls). If $x'$ is not stable, or if $x'$ is stable but $R$ is not applicable to $x'$, then the upper bound is updated as $\nu:=\mu$. And if $x'$ is stable and $R$ is applicable to $x'$, then the lower bound is updated as $\lambda:=\mu$.  Then we handle in a similar way the updated interval $[\lambda,\nu]$. And so on until we get $\lambda,\nu$ such that $\nu-\lambda\le 1$.
 \end{itemize}

Upon termination of~(P), $\lambda$ or $\lambda+1$ is exactly $\tau_L(x)$, the number of iterations is at most $\log \bmax$, and each iteration takes time polynomial in $|V|$ (including oracle calls). This together with Theorem~\ref{tm:condC} and explanations in~Sect.~\SSEC{constr_poset} implies that
 \begin{numitem1} \label{eq:complex_C}
subject to condition~(C), if $\xmin$ is known, then the weighted rotational poset $(\Uscr,\tau,\lessdot)$ can be constructed in weakly polynomial time, namely, in time bounded by $\log \bmax$ times a polynomial in $|V|$.
 \end{numitem1}
 
It remains to explain how to find the initial stable g-matching $\xmin$ when (C) is valid so as to get a weakly polynomial time bound as well. The method from~\cite{AG} reviewed in Sect.~\SSEC{xmin} guarantees merely pseudo polynomial time. Below we elaborate another method to find $\xmin$. It is described for a general SGMM, without imposing~(C).
\medskip

\noindent\textbf{Alternative method of finding $\xmin$.}
It consists of two stages.
\smallskip

\noindent\textbf{\emph{Stage 1}.} ~The goal is to construct some g-matching which is stable but may differ from $\xmin$. This falls into iterations. 

In the input of a current iteration, we are given $x\in\Zset_+^E$ such that $x\le b$ and $x$ satisfies the next two requirements. The first one is that
  \begin{numitem1} \label{eq:first_req}
 $x$ is acceptable, i.e. $C_v(x_v)=x_v$ holds for all $v\in V$.
 \end{numitem1}
 
Like in Sect.~\SSEC{act_graph}, let $U^+_F(x)$ (resp. $U^-_F(x)$) denote the set of edges $e=wf$ (where $f\in F$) that are interesting (resp. not interesting) for $f$ under $x$. We allow the existence of blocking edges for $x$ in a restricted form, by imposing the second requirement:
 \begin{numitem1}\label{eq:second_req}
for each edge $e=wf\in U^+_F(x)$, if $e$ is interesting for $w$ under $x$, then $C_w(x_w+\onebf_w^e)$ is equal to $x_w+\onebf_w^e$ (rather than $x_w+\onebf_w^e-\onebf^{e'}$ for some $e'\in E_w-\{e\}$).
 \end{numitem1}

Initially we put $x:=0$ (unlike the algorithm in~\SSEC{xmin} where the initial function is equal to $b$); clearly this $x$ obeys both~\refeq{first_req} and~\refeq{second_req}. Next we describe how to handle the current (input) vector $x$ at an iteration. Two cases are possible.
 \medskip
 
\noindent\underline{\emph{Case 1}}. Each edge $e=wf\in U^+_F(x)$ is not interesting for $w$ under $x$. Then $x$ admits  no blocking edges; therefore, it is stable, and we are done. 
  \medskip
  
\noindent\underline{\emph{Case 2}}. There is $e=wf\in U^+_F(x)$ such that $C_w(x_w+\onebf_w^e)=x_w+\onebf_w^e$. We fix one of such edges, $e'=w'f'$ say, and construct the maximal \emph{edge-simple} path $P=(v_0,e_1,v_1,\ldots, e_k,v_k)$ beginning with $v_0=w'$ and $e_1=e'$ such that: each $(e_i,e_{i+1})$ is a legal $F$-pair when $v_i\in F$ (i.e. $i$ is odd), and a quasi-essential $W$-pair when $v_i\in W$. Hereinafter for $w\in W$ and $c,a\in E_w$, we call $(c,a)$ a \emph{quasi-legal} $W$-pair for $x$ if
  \begin{numitem1} \label{eq:quasi-leg}
(i) $c\in U_F^-(x)$ and $a\in U_F^+(x)$; (ii) $a$ is not interesting for $w$ under $x$; and (iii) the vector $z:=x_w-\onebf^c+\onebf^a$ is acceptable, i.e. $C_w(z)=z$; 
\end{numitem1}
and call $(c,a)$ a \emph{quasi-essential} $W$-pair for $x$ if, in addition,
  \begin{numitem1} \label{eq:quasi-ess}
for any $d\in E_w-\{a\}$, if $(c,d)$ is another quasi-legal $W$-pair, then the edge $d$ is not interesting for $w$ under $z:=x_w-\onebf^c+\onebf^a$.
  \end{numitem1}
(We prefer to add the prefix ``quasi'' to formally differ from Definitions~3,4 in Sect.~\SSEC{act_graph} where the stability of $x$ implies that~(ii) in~\refeq{quasi-leg} is valid automatically.) Like property~\refeq{U+a} for essential pairs, one shows (by arguing as in the proof of Lemma \ref{lm:essent_pair}) that for $c$ fixed, the quasi-essential $W$-pair $(c,a)$ (if exists) is determined uniquely.

As a consequence, the path $P$ as above is constructed uniquely. Three subcases are possible: 

(a) the last vertex $v'':=v_k$ belongs to $F$ and the last edge $e'':=e_k$ (which is in $U^+_F(x)$) does not have legal $F$-pair, i.e. $C_{v''}(x_{v''}+\onebf_{v''}^{e''})=x_{v''}+\onebf_{v''}^{e''}$; 

(b) $v''$ belongs to $W$ and the edge $e''$ (which is in $U^-_F(x)$) has no quasi-legal $W$-pair (whence it does not have quasi-essential $W$-pair either); 

(c) $v_k$ coincides with some $v_i$ ($i<k$) and, moreover, the part of $P$ from $v_i$ to $v_k$ forms a cycle  $R$, called a \emph{quasi-rotation}, in which the pair $(e'',e_i)$ is a legal $F$-pair when $v_k\in F$, and a quasi-essential $W$-pair when $v_k\in W$.  

In subcases (a),(b), we  shift $x$ along $P$ with weight 1, i.e. transform $x$ into $y:=x+\chi^{P}$, where $\chi^P(e)$ takes value 1 ($-1$) if $e\in P^+$ (resp. $e\in P^-$), and 0 otherwise. And in subcase~(c), we shift $x$ along the quasi-rotation $R$ with weight 1. Below (in Lemma~\ref{lm:quasi-shift}) we show that the obtained function, $y$ say, satisfies:
 \begin{numitem1} \label{eq:correct_y}
(i) $y$ is acceptable; (ii) $y_f\succeq_f x_f$ for all $f\in F$;  and (iii) $y$ obeys~\refeq{second_req}. 
  \end{numitem1}
  
 For convenience, we further refer to paths $P$ and cycles (quasi-rotations) $R$ as above as \emph{augmenting} paths and cycles.

So, in Case~2, the current iteration consists in: choosing an edge $e=wf\in U^+_F(x)$ satisfying $C_w(x_w+\onebf_w^e)=x_w+\onebf_w^e$, constructing the corresponding augmenting path or cycle, and shifting $x$ along this path/cycle with weight 1. In view of~\refeq{correct_y}(ii), the sequence of iterations is finite and the process terminates as soon as the current vector, $x'$ say, becomes free of edges $e$ as above (w.r.t. $x'$). Then Case~1 takes place, implying that $x'$ is a stable g-matching, as required. This completes Stage~1.
 \medskip
 
\noindent\textbf{\emph{Stage 2}.}
Here the goal is to transform the stable g-matching $x$ found at Stage~1 into $\xmin$. To do so, we exchange the roles of $F$ and $W$, considering the reversed lattice $(\Lscr,\prec_W)$, and construct a principal route $\Tscr$ going from $x'$ to the maximal element of this lattice, which is just $\xmin$. 
\medskip

The above algorithm is pseudo polynomial. In order to accelerate it, subject to condition~(C), at Stage~2 we compute the maximal weighs of rotation by use of procedure~(P) and accordingly apply rotations with these weights. And at Stage~1, we act similarly to compute the maximal feasible weights of augmenting paths/cycles found on iterations of this stage and then apply them with these weights. Arguing as in the proof of Theorem~\ref{tm:condC}, one shows that during Stage~1 all augmenting paths/cycles are different and their number is bounded by a polynomial in~$|V|$ (we omit details here). This together with~\refeq{complex_C} and Theorem~\ref{tm:condC} gives the following result.
   \begin{theorem} \label{tm:complex_condC}
Subject to condition~(C), in the weighted rotational poset $(\Uscr,\tau,\lessdot)$, the size $|\Uscr|$ does not exceed $4|E|^2|F| |W|$, and the poset can be constructed in weakly polynomial time, namely, the number of oracle calls and usual operations in the above (combined) algorithm is bounded by $\log \bmax$ times a polynomial in $|V|$.
 \end{theorem}
 
We finish this section with the next lemma that verifies the correctness of Stage~1 of the algorithm.

\begin{lemma} \label{lm:quasi-shift}
The vector $y$ obtained from $x$ at a current iteration satisfies~\refeq{correct_y}.
 \end{lemma}
  \begin{proof}
~It is outlined in a sketched manner, leaving some details to the reader. We apply reasonings similar or close to ones in Sect.~\SEC{act-rot} (in particular, in the proofs of Lemmas~\ref{lm:acac} and~\ref{lm:partW}). On this way, one can realize that the lemma is reduced to the following assertions, where a vector $x'\in\Zset_+^E$ satisfies~\refeq{first_req} and~\refeq{second_req}.
\begin{itemize}
\item[(A)] For $f\in F$ and $a,c\in E_f$, let $(a,c)$ be a legal $F$-pair for $x'$, and let $x''$ be obtained by applying $(a,c)$ to $x'$ (i.e. $x''_f=x'_f+\onebf_f^a-\onebf_f^c$). Then: (a1) $x''_f\succ x'_f$; (a2) if $a'\in E_f-\{a\}$ is interesting for $f$ under $x'$, then so is under $x''$; and (a3) all legal $F$-pairs $(a',c')$ for $x'$ with $a'\ne a$ and $c'\ne c$ are legal for $x''$ as well.
\item[(B)] For $w\in W$ and $c,a\in E_w$, let $(c,a)$ be a quasi-essential $W$-pair, and let $x''$ be obtained by applying $(c,a)$ to $x'$ (i.e. $x''_w=x'_w-\onebf_w^c+\onebf_w^a$). Then: (b1) all quasi-essential $W$-pairs $(c',a')$ for $x'$ with $c'\ne c$ and $a'\ne a$ are quasi-essential for $x''$ as well; and (b2) if $d=wf\in E_w-\{a,c\}$ is interesting for both $f$ and $w$ under $x''$, then $C_w(x''_w+\onebf_w^d)=x''_w+\onebf_w^d$. Also similar assertions hold for $x':=x$ and $a:=e_1$ (ignoring $c$), where $e_1$ is the first edge in the path $P$.
  \end{itemize}
  
The proofs of (a1)--(a3) and (b1) are routine in essence (as being close to corresponding assertions from Sect.~\SEC{act-rot}). The crucial assertion is (b2); we prove it as follows. 

Let $C:=C_w$ and $\onebf^\bullet:=\onebf_w^\bullet$. Consider $z:=x''_w$ ($=x'_w-\onebf^c+\onebf^a$). Suppose that there is $d\in E_w-\{c,a\}$ violating (b2), i.e. $d\in U_F^+(x'')$ and $C(z+\onebf^d)=z+\onebf^d-\onebf^{c'}$ for some $c'$. Note that $d\in U_F^+(x'')$ implies $d\in U_F^+(x')$. Then either (i) $C(x'_w+\onebf^d)=x'_w$, or (ii) $C(x'_w+\onebf^d)=x'_w+\onebf^d$ (by~\refeq{second_req} applied to $x=x'$ and $e=d$). 

In Case~(i), let $z':=x-\onebf^c+\onebf^d$. Then $z+\onebf^d>z'$, and applying~(A2), we have $C(z+\onebf^d)\wedge z'\le C(z')$, which implies $C(z')=z'$. It follows that $(c,d)$ is quasi-legal for $x'$, in view of~(i). But then, by~\refeq{quasi-ess}, $d$ is not interesting for $w$ under $z$; a contradiction.

So we are in case~(ii). Let $\tilde z:=x'_w+\onebf^a+\onebf^d$. Then $\tilde z> x'_w+\onebf^d$ implies $|C(\tilde z)|\ge |x'_w+\onebf^d|$, by~(A3). In view of $C(x'_w+\onebf^a)=x'_w$, we have $C(\tilde z)=C(x'_w+\onebf^d)=x'_w+\onebf^d$. On the other hand, $\tilde z>z+\onebf^d$ implies $C(\tilde z)\wedge (z+\onebf^d)\le C(z+\onebf^d)=z+\onebf^d-\onebf^{c'}$. But $C(\tilde z)(c')=z(c')=x'(c')$, whereas  $C(z+\onebf^d)(c')=x'(c')-1$; a contradiction.

This completes the proof of the lemma.
\end{proof}


\section{Concluding remarks} \label{sec:concl}

We finish this paper with a brief outline of two related issues.
\medskip

\textbf{I.} As a natural generalization of SGMM, one can address the \emph{problem of finding a stable g-matching of minimum cost}: 
  \begin{numitem1} \label{eq:ast}
Given \emph{costs} $c(e)\in\Zset$ of edges $e\in E$, find $x\in\Sscr$ minimizing the total cost $cx:=\sum(c(e)x(e)\colon e\in E)$. 
  \end{numitem1}
  
By Theorem~\ref{tm:isomorph_complete}, $\Sscr$ is represented via the set of closed functions for the poset $\Pscr=(\Uscr,\tau,\lessdot)$, and we know that in a general case of SGMM, the size $|\Uscr|=N$ of this poset is estimated as $O(\bmax|E|)$ (by Theorem~\ref{tm:gen_time}  and Lemma~\ref{lm:N}). This enables us to solve problem~\refeq{ast} in time polynomial time in $O(\bmax|E|)$, whenever the poset $\Pscr$ is explicitly known. (This poset can be found in time polynomial in $\bmax,N,|V|$; cf. Theorem~\ref{tm:gen_time}.) To this aim, we apply an approach originally elaborated in Irving et al.~\cite{ILG} to minimize a linear function over stable marriages and subsequently extended to some other models of stability. It consists in a reduction to the standard minimum cut problem, by a method due to Picard~\cite{pic}. A brief outline of how it works in our case of SGMM is as follows.

First we compute the cost $c_R:=c(R^+)-c(R^-)$ of each rotation $R\in\Rscr$. Then for each $x\in\Sscr$ and the corresponding closed function $\zeta=\phi(x)$ on $\Uscr$ (cf. Definition~9), we have $cx=c\xmin+\sum(c_R\zeta(R^{(i)})\colon R^{(i)}\in \Uscr)$ (where $R^{(i)}$ is a labeled copy of $R$ in $\Uscr$).
This gives the problem equivalent to~\refeq{ast}, namely:
 \begin{numitem1} \label{eq:astast}
Find a closed function $\zeta$ minimizing the weight $c^\zeta:=\sum(c_R\zeta(R^{(i)})\colon R^{(i)}\in\Uscr)$. 
  \end{numitem1}
Clearly \refeq{astast} has an optimal solution determined by an ideal $\Iscr$ of $(\Uscr,\lessdot)$ (in this case $\zeta(R^{(i)})=\tau_{R^{(i)}})$ if $R^{(i)}\in \Iscr$, and 0 otherwise).

To solve \refeq{astast}, we act as in~\cite{pic} and transform the Hasse diagram $H=(\Uscr,\Escr)$ of the poset $\Pscr$ into directed graph $\hat H=(\hat\Uscr,\hat \Escr)$ with edge capacities $h$, as follows:

(a) add two vertices: ``source'' $s$ and ``sink'' $t$; 

(b) add the set $\Escr^+$ of edges $(s,R^{(i)})$ for all $R^{(i)}\in\Uscr^+:=\{R^{(i)}\in \Uscr\colon c_{R^{(i)}}>0\}$;

(c) add the set $\Escr^-$ of edges $(R^{(i)},t)$ for all $R^{(i)}\in\Uscr^-:=\{R^{(i)}\in\Uscr\colon c_{R^{(i)}}<0\}$;

(d) assign the capacities $h(s,R^{(i)}):=c_{R^{(i)}}\tau_{R^{(i)}}$ for $(s,R^{(i)})\in \Escr^+$,  $h(R^{(i)},t):=|c_{R^{(i)}}\tau_{R^{(i)}}|$ for $(R^{(i)},t)\in \Escr^-$, and $h(R^{(i)},{R'}^{(j)}):=\infty$ for $(R^{(i)},{R'}^{(j})\in\Escr$.
 \smallskip
 
Now it is rather straightforward to show that an ideal $\Iscr$ of $(\Uscr,\lessdot)$ determines an optimal solution to \refeq{astast} if and only if the set of edges of $\hat H$ going from $\{s\}\cup(\Uscr-\Iscr)$ to $\{t\}\cup\Iscr$ forms a minimum capacity cut among those separating $s$ from $t$. 

Thus,~\refeq{ast} is reduced to the minimum $s$--$t$ cut problem in the graph $\hat H=(\hat\Uscr,\hat \Escr)$ with nonnegative integer edge capacities $h$, for which there is known plenty of solution algorithms of complexity $O(n^3)$ and faster (where $n$ is the number of vertices of a graph). Summarizing explanations above, we obtain the following
  \begin{prop} \label{pr:cost_probl}
{\rm(i)} When the poset $(\Uscr,\tau,\lessdot)$ is available, problem~\refeq{ast} is solvable in time polynomial in $|\Uscr|$, or $\bmax|E|.$

{\rm(ii)} Given $G,b,C$ as above and $c\in\Zset^E$, if the gapless condition~(C) is valid, then  problem~\refeq{ast} is solvable in time $\log \bmax$ times a polynomial in $V$, including oracle calls for CFs (where the factor of $\log \bmax$ arises when computing weights $\tau$).
\end{prop}


\textbf{II.} By Lemma~\ref{lm:N}, the length $N$ of the full route does not exceed $\bmax|E|/2$. As is shown in~\cite{karz3}, there is a ``small'' graph $G=(V,E)$ with ``big'' capacities $b$ such that $N=\bmax$. We briefly outline this construction, referring for details to~\cite{karz3}. The graph $G$ is drawn in the left fragment of  the picture below.

\vspace{-0.2cm}
\begin{center}
\includegraphics[scale=0.8]{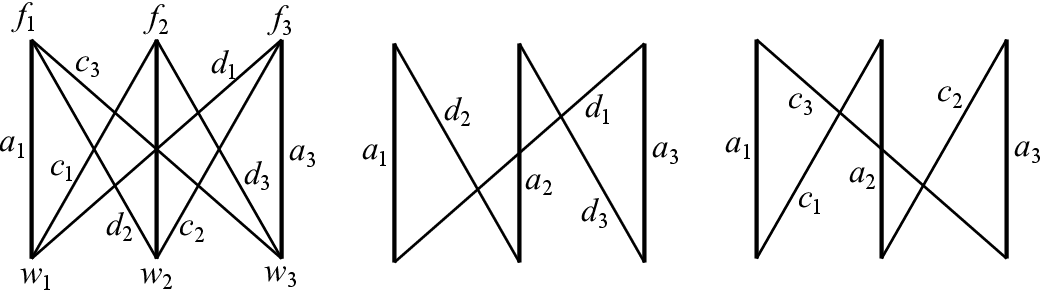}
\end{center}
\vspace{-0.3cm}

Here:
  \begin{numitem1}\label{eq:ex3x3}
\begin{itemize}
\item[(i)]
the worker set $W$ is $\{w_1,w_2,w_3\}$ and the firm set $F$ is $\{f_1,f_2,f_3\}$; for $i=1,2,3$, the set $E_{w_i}$ consists of three edges $w_if_i$, $w_if_{i+1}$, $w_if_{i-1}$, denoted as $a_i,c_i,d_i$, respectively (taking indexes modulo 3);
  \item[(ii)]
the capacities $b$ of vertical (bold) edges $w_if_i$ ($i=1,2,3$) are equal to the same even integer $q=2p$, and the capacities of other (thin) edges $w_if_j$, $i\ne j$, to the same number $p$; for each vertex $v$, the sum of values on $E_v$ is bounded by the quota $q$, and $\Bscr_v:=\{z\in \Zset_+^{E_v}\colon z(e)\le b(e)\; \forall e\in E_v\}$;
  \item[(iii)]
for each $w_i$, the choice function $C_{w_i}$ is subject to the linear order 
$c_i>d_i>a_i$ (cf.~Example~1 in Sect.~\SEC{defin});
  \item[(iv)]
for each $f_i$, the choice function $C:=C_{f_i}$ is defined by the following rule:  
for $z\in \Bscr_{f_i}$, (a) if $|z|\le q$, then $C(z)=z$; and (b) if $|z|>q$, then  $C(z)(a_i):=z(a_i)$, $C(z)(c_{i-1}):=\zeta$ and $C(z)(d_{i+1}):=\delta$, where $\zeta,\delta$ are such that $\zeta\le z(c_{i-1})$, ~$\delta\le z(d_{i+1})$, ~$z(a_i)+\zeta+\delta =q$, and there hold:
 \begin{itemize}
\item[(b1)]   $|\zeta-\delta|$ is minimum; and
\item[(b2)] subject to~(b1), if $\zeta<\delta$, then $z(c_{i-1})=\zeta$.
  \end{itemize}
 \end{itemize}
    \end{numitem1}

One shows that the CFs $C_{f_i}$ satisfy axioms~(A2),(A4). Also $\bmax=q$. We construct the sequence of g-matchings $x^0,x^1,\ldots, x^q$ as follows:
 \begin{numitem1} \label{eq:q_seq}
for $k=0,\ldots,q$ and $i=1,2,3$,

$x^k(a_i):=k$; for $k$ even, $x^k(c_i)=x^k(d_i):=p-i/2$; and for $k$ odd, 
$x^k(c_i):= p-(i-1)/2$, and $x^k(d_i):=p-(i+1)/2$.
 \end{numitem1}

In particular, $x^0(a_i)=0$ and $x^0(c_i)=x^0(d_i)=p$ for $i=1,2,3$, which is
nothing else than the initial (the best for $W$) stable g-matching $\xmin$. Moreover, one can check that: (a) each $x^k$ is stable, and $x^q=\xmax$ (taking the value $q$ on each $a_i$, and 0 on the $c_i$ and $d_i$); (b) the sequence $x^0,x^1,\ldots,x^q$ is the \emph{unique} full route $\Tscr$; (c) for $k$ even, $x^{k+1}$ is obtained from $x_k$ by shifting with unit weight along the rotation $R$ formed by the sequence of edges $a_1,d_2,a_2,d_3,a_3,d_1$ (where $R^+=\{a_1,a_2,a_3\}$), illustrated in the middle fragment of the above picture; and (d) for $k$ odd, $x^{k+1}$ is obtained from $x^k$ by shifting with unit weight along the rotation $R'$ formed by the sequence of edges $a_1,c_3,a_3,c_2,a_2,c_1$ (where $R'^+=\{a_1,a_2,a_3\}$), illustrated in the right fragment of the above picture. So $R$ and $R'$ alternate, and there is no other rotation used in $\Tscr$.

Thus, the corresponding rotational poset is viewed as a chain consisting of $q/2=\bmax/2$ copies of $R$ and $q/2$ copies of $R'$ (which alternate and never commute). This poset has $q+1$ closed sets, which correspond to the stable g-matchings $x^0,\ldots,x^q$. The even parameter $q>0$ can be chosen arbitrarily, and we conclude with the following
  \begin{corollary} \label{cor:Cnot}
In SGMM, if condition~(C) is not imposed, then there is a series of
instances with a fixed graph $G$ such that in the poset representations of the lattices of stable g-matchings, the posets have pseudo polynomial sizes (proportional to $\bmax$).
 \end{corollary}

\end{document}